\documentclass[a4paper,10pt,reqno]{amsart}
\usepackage{amssymb, amsmath, graphicx, amsthm, enumitem, multicol,amsfonts}
\usepackage[margin=1in]{geometry}
\usepackage{cite}
\usepackage{ bbold }
\usepackage{longtable}
\usepackage{booktabs}
\usepackage{placeins}
\usepackage{hyperref}
\usepackage{tikz-cd}

\usepackage[utf8]{inputenc}
\theoremstyle{definition}

\newtheorem {proposition}   {Proposition}
\newtheorem {theorem}    {Theorem}

\DeclareMathOperator{\C}{\mathbb{C}}

\DeclareMathOperator{\Z}{\mathbb{Z}}
\DeclareMathOperator{\N}{\mathbb{N}}

\DeclareMathOperator{\cH}{\mathcal{H}}
\DeclareMathOperator{\cC}{\mathcal{C}}

\newcommand{\Tr}{\mathrm{Tr}\,}

\newcommand{\ba}{\begin{align}}
\newcommand{\ea}{\end{align}}
\newcommand{\bea}{\begin{eqnarray}}
\newcommand{\eea}{\end{eqnarray}}
\newcommand{\be}{\begin{equation}}
\newcommand{\ee}{\end{equation}}

\numberwithin{equation}{subsection}

\title{Modularity of Vertex Operator Algebra Correlators with Zero Modes}
 \author{Darlayne Addabbo}
 \address{Darlayne Addabbo, Department of Mathematics and Physics, SUNY Polytechnic Institute, Utica, NY 13502} \email{addabbd@sunypoly.edu}
 \author{ Christoph A.~Keller}
 \address{Christoph A. Keller,  Department of Mathematics, University of Arizona, Tucson, AZ 85721-0089, USA}
 \email{cakeller@arizona.edu}

\theoremstyle{plain}
\newtheorem{thm}{Theorem}[section]
\newtheorem{lem}[thm]{Lemma}
\newtheorem{prop}[thm]{Proposition}
\newtheorem{cor}[thm]{Corollary}
\newtheorem*{theorem-non}{Theorem}

\theoremstyle{definition}
\newtheorem{defn}{Definition}[section]

\theoremstyle{remark}

\begin{document}

 \begin{abstract}
 It is known from Zhu's results that under modular transformations, correlators of rational $C_2$-cofinite vertex operator algebras transform like Jacobi forms.
 We investigate the modular transformation properties of VOA correlators that have zero modes inserted. We derive recursion relations for such correlators and use them to establish modular transformation properties. We find that correlators with only zero modes transform like quasi-modular forms, and mixed correlators with both zero modes and vertex operators transform like quasi-Jacobi forms. As an application of our results, we introduce algebras of higher weight fields whose zero mode correlators mimic the properties of those of weight 1 fields. We also give a simplified proof of the weight 1 transformation properties originally proven by Miyamoto.
\end{abstract}

\maketitle

\section{Introduction}
Given a vertex operator algebra $V$ of central charge $c$ \cite{FLM}, \cite{MR1142494}, \cite{LL}, one can define correlation functions on the torus using traces of the form
\be\label{introcorr}
\Tr_M Y(\zeta_1^{L_0}a^1,\zeta_1) \ldots Y(\zeta_n^{L_0} a^n,\zeta_n) q^{L_0-c/24}\ , \qquad \zeta_i = e^{2\pi i z_i}\ ,
\ee
where $M$ is some module of the VOA and the $Y(\zeta^{L_0} a,\zeta)$ are vertex operators inserted at position $\zeta$. Zhu  \cite{Z} famously established that if $V$ is rational and $C_2$-cofinite, then these correlation functions transform nicely under modular transformations. To be more precise, they are weakly holomorphic vector valued Jacobi forms \cite{MR781735} of weight $k$ and index 0, where $k$ is given by the sum of the $L[0]$ weights of the states $a^i$. This means that they satisfy two properties: First, they are elliptic functions in the arguments $z_i$, that is they are invariant under the transformations $z_i \mapsto z_i + \lambda_i\tau + \mu_i$ for integer $\lambda_i,\mu_i$. This simply means that they are well-defined functions for the coordinates $z_i$ on the torus. Second, under modular transformations, $\gamma\tau = \frac{a\tau+b}{c\tau+d}$ and $\gamma z_i=\frac{z_i}{c\tau+d}$ where $\gamma=\begin{pmatrix}a&b\\c&d\end{pmatrix}\in SL(2,\Z)$, they pick up the usual factor $(c\tau+d)^k$. In particular, if $n=0$, that is for zero point correlators, we end up with (vector valued) modular forms.

Equation (\ref{introcorr}) suggests an immediate generalization: we can insert any number of zero modes in the trace to obtain a function
\be\label{intromixcorr}
\Tr_M b^1_0 \cdots b^m_0 Y(\zeta_1^{L_0}a^1,\zeta_1) \ldots Y(\zeta_n^{L_0} a^n,\zeta_n) q^{L_0-c/24}\ .
\ee
Here we use the physics notation $b_0=o(b)$ for the zero mode of the state $b$.
We will call this a \emph{mixed correlator}. If $n=0$ we will call it a \emph{zero mode correlator}, and if $m=0$ a \emph{full correlator}. We expect these correlators to also have good modular transformation properties.
The goal of this article is to investigate what exactly those properties are. 

For notational simplicity, we will focus on holomorphic VOAs, so that $M=V$, and assume that the central charges are multiples of 24. We will often suppress writing $M$ in denoting our correlators, as the trace is assumed to be taken over $V$ unless otherwise indicated. We believe that our results can be generalized in a straightforward way to rational $C_2$-cofinite VOAs, essentially by adding the term `vector valued' in front of all modular objects.

For mixed correlators we find first that the functions are no longer elliptic, but pick up an elliptic anomaly under $z\mapsto z+\tau$. This is not surprising and an immediate consequence of the fact that the vertex operators $Y(\zeta^{L_0}a,\zeta)$ do not necessarily commute with the zero modes $b_0$. Second, we find that the correlators also pick up an anomaly under modular transformations. Taken together this implies that they are (weakly holomorphic) \emph{quasi-Jacobi forms} \cite{MR2796409,MR3248159}.

Roughly speaking, quasi-Jacobi forms are to Jacobi forms what quasi-modular forms are to modular forms.
In particular, zero mode correlators are quasi-modular forms. 
Quasi-modular forms are well known \cite{MR1363056}. One standard example is the Eisenstein series $G_2(\tau)$, which transforms as
\be\label{G2}
(c\tau+d)^{-2}G_2(\frac{a\tau+b}{c\tau+d})=G_2(\tau)-\frac{2\pi i c}{c\tau+d}\ 
\ee for any $\begin{pmatrix}a&b\\c&d\end{pmatrix}\in SL(2,\Z),$ see \eqref{Eisenstein2} below for the definition.
Another standard example are $\tau$-derivatives of modular forms. Both these examples will play an important role in our work. Transforming a general quasi-modular form gives the same structure as (\ref{G2}): a head term, which is the term expected for a modular form of that weight, and a tail of terms encoding the modular anomaly.

To prove the above results we introduce recursion relations. These recursion relations allow us to express correlators in terms of correlators with fewer vertex operators inserted. Such recursion relations were first introduced in \cite{Z} for the case of correlators with no zero modes. \cite{Gaberdiel:2012yb} then derived a generalization for powers of a single zero mode inserted. In this article we give a recursion relation for the most general case (\ref{intromixcorr}) of different, possibly non-commuting zero modes $b_0^i$. See also \cite{MR4082238} for other uses of such recursion relations.

We obtain our result for mixed correlators by applying the recursion relation to (\ref{intromixcorr}), reducing the number of zero modes inserted until we are left with no zero modes. In the process, we pick up functions $g_j^{i}(z,\tau)$ defined in line \eqref{definition_g_j^i} below. The result then follows from Zhu's original theorem and the fact that the $g_j^{i}(z,\tau)$ are quasi-Jacobi forms. The result for zero mode correlators is simply a specialization to functions that do not depend on $z$.

Beyond investigating the general structure of mixed and zero mode correlators, we have a concrete question in mind. It is well known that the zero mode correlators of a state of weight 1 have nice modular transformation properties. 
More precisely, let $a$ be a weight 1 state. Then by the results of \cite{MR1738180}, see also \cite{MR2925472}, the generating function for the correlator with $n$ zero modes $a_0$ is a Jacobi form of weight 0. That is, defining
\be
\chi(\tau,z) := \Tr e^{2\pi iz a_0}q^{L_0-c/24}
= \sum_{n=0}^\infty \frac{(2\pi iz)^n}{n!} \Tr (a_0)^n q^{L_0-c/24}\ ,
\ee
$\chi(\tau,z)$ transforms as
\be\label{weight1generating}
\chi(\gamma\tau, \gamma z) = e^{\frac{\pi i \langle a,a\rangle c z^2}{c\tau+d}} \chi(\tau,z)\ .
\ee
A natural question is: What happens if we replace $a$ by a higher weight state? Is the generating function still some modular object? There has been work on this for free boson and lattice VOAs \cite{MR1827085}, for the $W_3$ algebra \cite{Iles:2014gra}, and for KdV charges \cite{Downing:2021mfw,Downing:2023lnp, Downing:2023lop,Downing:2024nfb}. For the general case this is still an open question.
Beyond its obvious mathematical interest, this question is also of great interest to physicists in the context of higher spin holography, see for example \cite{Gutperle:2011kf,Gaberdiel:2013jca,Castro:2016ehj}.

In this article we take steps towards an answer. In particular we will give conditions on the higher weight states inserted that are necessary to mimic (\ref{weight1generating}).
Our approach is the following.
By expanding (\ref{weight1generating}) in $z$ and comparing terms, we find the following transformation property for the zero mode correlators:
\be\label{weight1zeromodes}
(c\tau+d)^{-s} \Tr (a_0)^s e^{2\pi i\gamma\tau(L_0-c/24)} = \sum_{k=0}^{s/2} \left(\frac{c \langle a,a\rangle}{(c\tau+d)2\pi i}\right)^k \frac{s!}{2^k k! (s-2k)!} \Tr (a_0)^{s-2k} e^{2\pi i\tau (L_0-c/24)}. 
\ee
As expected from Theorem \ref{main}, these zero mode correlators are indeed quasi-modular forms. More importantly, they satisfy  two additional properties: First, as quasi-modular forms they are homogeneous of weight $s$. Second, their tail is a sum of lower point zero mode correlators. These two properties allow them to be summed up into an exponential function that transforms as (\ref{weight1generating}).

To repeat this for higher weight states, we want to impose conditions on the fields so that the zero mode correlators mimic the above properties of the correlation functions of the Heisenberg algebra. To this end we define what we call a \emph{higher weight} or \emph{heavy Heisenberg algebra (HHA)}. We show that its zero mode correlators satisfy two properties similar to (\ref{weight1zeromodes}): First, they are quasi-modular forms of homogeneous weight. We achieve this by requiring that the zero modes of the HHA fields commute. Second, in the recursion relations only zero mode correlators of fields in the HHA should appear. We achieve this by requiring that the HHA is closed in an appropriate sense. 
In section~\ref{s:sVOA} we introduce the notion of such an HHA  and show that it indeed leads to zero mode correlators with the right properties. In section~\ref{s:examples} we give examples of such HHA for weight 1 and 2 fields. The weight 1 example was of course already treated in \cite{MR1738180}. However, we give a shorter proof of (\ref{weight1generating}). We also give a general example of a weight 2 HHA, and then work out its correlators explicitly for the case of lattice VOAs by using the results of \cite{MR1827085}.

\subsection*{Acknowledgements}
We thank Alejandra Castro, Matthias Gaberdiel, and Tom Hartman for comments on the draft. 
We thank two anonymous referees for very detailed and helpful feedback and for pointing out a gap in the original draft. The work of CAK is supported by NSF Grant 2111748.

\section{Basics}

\subsection{Quasi-modular forms and quasi-Jacobi forms}

Let us first gather some definitions and examples of modular objects that play an important role in computing our correlators.
The following definitions are what is needed to describe correlators of holomorphic VOAs. To deal with rational $C_2$-cofinite VOAs, one would introduce vector valued versions of the modular objects below.
Other than the usual modular forms for $SL(2,\Z)$, we will also need \emph{weakly holomorphic} modular forms, that is forms that are holomorphic on the upper half plane with a pole at the cusp $\tau=i\infty$.

First we define a generalization of modular forms to quasi-modular forms \cite{MR1363056}:

\begin{defn}Let $\mathcal{H}$ denote the complex upper half plane. Let $f:\mathcal{H}\rightarrow \mathbb{C}$ be a holomorphic function, $k$ and $s\ge0$ be integers. Then $f$ is called a {{\it quasi-modular function}} of weight $k$ and depth $s$ if there exist holomorphic functions $f_0, \cdots, f_s$ over $\mathcal{H}$ with $f_s$ not identically $0$, such that
\begin{equation}
(c\tau+d)^{-k}f\left(\frac{a\tau+b}{c\tau+d}\right)=\sum_{j=0}^sf_j(\tau)\left(\frac{c}{c\tau+d}\right)^j
\end{equation}for any $\begin{pmatrix}a&b\\c&d\end{pmatrix}\in SL(2,\Z)$ and any $\tau\in \mathcal{H}$. If all $f_j(\tau)$ grow at most polynomially as $\Im \tau \rightarrow 0$, then we say $f(\tau)$ is a \emph{quasi-modular form}, see \cite{MR3025137} for an introduction to the subject.
\end{defn}

Next we define weak Jacobi forms for multiple variables. These are a generalization of the Jacobi forms introduced in \cite{MR781735}. We follow \cite{MR3248159} here:
\begin{defn}
Let $F$ be a real symmetric $n\times n$ matrix. A function $\phi: \cH\times \C^n \to \C$ holomorphic in $\cH$ and meromorphic in $\C^n$ is a \emph{weak Jacobi form of weight $k$, index $F$} and trivial character  if it has an expansion of the form
\be\label{Jfexpansion}
\phi(\tau,\underline z)=\sum_{\ell\in \mathbb{N}, \underline{\nu}\in \Z^n } c(\ell,\underline{\nu}) q^\ell \exp(2\pi i (\underline z^T \underline\nu))\ ,
\ee
and for all $\gamma=\begin{pmatrix}a&b\\c&d\end{pmatrix}\in SL(2,\Z)$ and $(\underline \nu,\underline \mu)\in \Z^n \times \Z^n$ we have
\be
\phi(\gamma\tau,\gamma \underline z) = (c\tau+d)^k \exp \left( 2\pi i \frac{c F[ \underline z ]}{c\tau+d} \right) \phi(\tau, \underline z)
\ee
and
\be
\phi(\tau,\underline z+\underline \lambda \tau+\underline \mu) = \exp (-2\pi i(\tau F[\underline\lambda]+2\underline z^T F \underline \lambda))\phi(\tau,\underline z)\ .
\ee
\end{defn}
Note that our definition differs slightly from the original definitions \cite{MR781735,MR3248159} insofar as we allow the functions to have poles in the $\underline z$. This allows for non-constant forms of index 0 such as Weierstrass functions. The reason to use this broader definition is that we want to describe VOA correlators, which do have poles in $\underline z$.

There is a corresponding notion of a weakly holomorphic Jacobi form by allowing for a pole of say order $p$ at $q=0$, that is by allowing $\ell$ to run from $-p$.
Jacobi forms will appear in two contexts: On the one hand, by the results of Zhu \cite{Z}, correlators of holomorphic VOAs with $c=24p$ are weakly holomorphic Jacobi forms of index 0. On the other hand, the generating function of zero-mode correlators of a weight 1 state is a Jacobi form of weight 0 \cite{MR1738180}.

Next we want to define quasi-Jacobi forms. Here we follow \cite{MR2796409} and again \cite{MR3248159}.

\begin{defn}
A function $\phi: \cH\times \C^n \to \C$ holomorphic in $\cH$ and meromorphic in $\C^n$ is a \emph{weak quasi-Jacobi  form of weight $k$, index $F$} and trivial character if it has an expansion (\ref{Jfexpansion}) and there are natural numbers $s_1,\ldots,s_n,t$ and meromorphic functions $S_{i_1,\ldots,i_n,j}(\phi)$ and $T_{i_1,\ldots i_n}(\phi)$ determined only by $\phi$ but independent of $\gamma$ and $\underline \mu, \underline \lambda$ such that
\begin{multline}\label{quasimodular}
(c\tau+d)^{-k}\exp \left(- 2\pi i \frac{c F[ \underline z ]}{c\tau+d} \right) \phi(\gamma\tau,\gamma\underline z) =\\
\sum_{\substack{i_1\leq s_1,\ldots i_n\leq s_n\\ j\leq t}} S_{i_1,\ldots ,i_n,j}(\phi)(\tau,\underline z) \left(\frac{cz_1}{c\tau+d}\right)^{i_1}\cdots\left(\frac{cz_n}{c\tau+d}\right)^{i_n} \left(\frac{c}{c\tau+d}\right)^{j}\ ,
\end{multline}
and
\be\label{quasielliptic}
\exp (2\pi i(\tau F[\underline\lambda]+2\underline z^T F \underline \lambda))\phi(\tau,\underline z+\underline \lambda \tau+\underline \mu) =\\ \sum_{\substack{i_1\leq s_1,\ldots i_n\leq s_n}} T_{i_1,\ldots, i_n}(\phi)(\tau,\underline z) \lambda_1^{i_1}\cdots \lambda_n^{i_n}\ .
\ee
Taking $S_{s_1,\ldots, s_n,t}(\phi)\neq 0$ or $T_{s_1,\ldots, s_n}(\phi)\neq 0$ we say $\phi$ has \emph{depth} $(s_1,\ldots,s_n,t)$.
\end{defn}

We note that weak quasi-Jacobi forms form a ring. Under multiplication the weights are additive and the depths are subadditive.  Moreover, this ring is closed under application of $\partial_{z_i},\partial_\tau$ and multiplication by $G_2(\tau)$.

\subsection{Examples}

In this section, we review some definitions and results which will be used throughout the paper. This material and proofs of the various statements can be found in Section $3$ of \cite{Z} and in \cite{MR890960}.

\begin{defn}
The Eisenstein series $G_{2k}(\tau)$ are the series 
\begin{equation}G_{2k}(\tau):=\sum_{(m,n)\in \Z^2\setminus\{(0,0)\}}\frac{1}{(m\tau+n)^{2k}}
\end{equation}for $k\ge 2$ and
\begin{equation}\label{Eisenstein2}G_2(\tau):=\frac{\pi^2}{3}+\sum_{m\in \mathbb{Z}\setminus\{0\}}\sum_{n\in \mathbb{Z}}\frac{1}{(m\tau+n)^2}.
\end{equation}
\end{defn}
Let $\xi(2k)=\sum_{n=1}^\infty\frac{1}{n^{2k}}$, $q=e^{2\pi i\tau}$, and $\sigma_k(n)=\displaystyle\sum_{\substack{d|n\\d>0}}d^k$. Then the $q$-expansion of the Eisenstein series $G_{2k}(\tau)$ is given by
\begin{equation}G_{2k}(\tau)=2\xi (2k)+\frac{2(2\pi i )^{2k}}{(2k-1)!}\sum_{n=1}^\infty\sigma_{2k-1}(n)q^n.
\end{equation}
As is well known, if $k\ge 2$, the series $G_{2k}(\tau)$ is a modular form of weight $2k$ for the modular group $SL(2,\Z)$. The series $G_2(\tau)$ is a quasi-modular form of weight 2 and depth 1. Its transformation under $SL(2,\Z)$ is given by \eqref{G2} above.

We next recall the definition of the Weierstrass $\wp$-function and its relatives.
\begin{defn}Let 
\begin{equation}
\wp_1(z,\tau)=\frac{1}{z}+\sum_{\substack{(m,n)\in \mathbb{Z}\oplus\mathbb{Z}\\(m,n)\in \Z^2\setminus\{(0,0)\}}}\left(\frac{1}{z-(m\tau+n)}+\frac{1}{m\tau+n}+\frac{z}{(m\tau+n)^2}\right)
\end{equation}
and
\begin{equation}
\wp_2(z,\tau)=\frac{1}{z^2}+\sum_{\substack{(m,n)\in \mathbb{Z}\oplus\mathbb{Z}\\(m,n)\in \Z^2\setminus\{(0,0)\}}}\left(\frac{1}{(z-(m\tau+n))^2}-\frac{1}{(m\tau+n)^2}\right).\end{equation} 
The function $\wp_2(z,\tau)$ is often simply called the \emph{Weierstrass $\wp$-function} and is denoted by $\wp(z,\tau)$. $\wp_1(z,\tau)$ is sometimes also called the \emph{Weierstrass zeta function} $\zeta(z,\tau)$ \cite{MR890960}. Both functions converge absolutely and uniformly on any compact subset not containing lattice points, so that we are allowed to take derivatives:
For $k>2$, define \begin{equation}\wp_{k+1}(z, \tau)=-\frac{1}{k}\frac{d}{dz}\wp_k(z,\tau).\end{equation}
\end{defn}
For $k\ge 1$, the Laurent expansion of $\wp_k(z,\tau)$ near $z=0$ is given by
\begin{equation}
\wp_k(z,\tau)=\frac{1}{z^k}+(-1)^k\sum_{n=1}^\infty\binom{2n+1}{k-1}G_{2n+2}(\tau)z^{2n+2-k}\ ,
\end{equation}
from which its modular transformation properties follow.
Let $\zeta=e^{2\pi i z}$ and $q=e^{2\pi i \tau}$. The functions $\wp_1(z,\tau)$ and $\wp_2(z,\tau)$ have $q$-expansions given by 
\begin{equation}\label{wp1}
\wp_1(z, \tau)=G_2(\tau)z+\pi i \frac{\zeta+1}{\zeta-1}+2\pi i \sum_{n=1}^\infty \left(\frac{q^n}{\zeta-q^n}-\frac{\zeta q^n}{1-\zeta q^n}\right)
\end{equation} 
and 
\begin{equation}\label{wp2}
\wp_2(z, \tau)=-\frac{\pi^2}{3}+(2\pi i)^2\sum_{m\in \mathbb{Z}}\frac{q^m\zeta}{(1-q 
\zeta)^2}-2(2\pi i)^2\sum_{n=1}^\infty\frac{n q^n}{1-q^n},\end{equation}
where the fractions in Equations \eqref{wp1} and \eqref{wp2} are expanded in nonnegative powers of $q$. 

\begin{defn}
For $k\ge 1$, let $P_k(\zeta,q)$ be the formal power series
\begin{equation}\label{Pk}
P_k(\zeta, q)=\frac{(2\pi i)^k}{(k-1)!}\sum_{n\in \Z\setminus\{0\}}\frac{n^{k-1}\zeta^n}{1-q^n}
\end{equation}
where $\frac{1}{1-q^n}=\sum_{i=0}^\infty q^{ni}$.
\end{defn}

\begin{prop}\label{prop:Pexpanded}
The formal power series $P_k(\zeta,q)$ converge uniformly and absolutely in every closed subset of the domain $\{(\zeta,q)||q|<|\zeta|<1\}$. The limit of $P_k(\zeta,q)$ which we still denote by $P_k(\zeta,q)$ satisfy the following formulas
\begin{equation}P_1(\zeta, q)=-\wp_1(z,\tau)+G_2(\tau)z-\pi i,
\end{equation}
\begin{equation}P_2(\zeta, q)=\wp_2(z,\tau)+G_2(\tau),
\end{equation} and for $k>2$
\begin{equation}P_k(\zeta, q)=(-1)^k\wp_k(z,\tau),
\end{equation}where $\zeta=e^{2\pi i z}$ and $q=e^{2\pi i \tau}$.
\end{prop}
We have 
\be
\zeta \frac{d}{d\zeta} P_k(\zeta,q) = \frac{k}{2\pi i} P_{k+1}(\zeta,q)\ .
\ee
For future convenience let us also define 
\be
\tilde P_1(\zeta,q) = P_1(\zeta,q)+\pi i\ .
\ee
In what follows, we will often slightly abuse our notation and write $P_k(z,\tau)$ for $P_k(\zeta,q)$.
Using proposition~\ref{prop:Pexpanded} and the elliptic properties of the $\wp_k(z,\tau)$ given in \cite{Z}, \cite{MR890960}, one easily verifies that the $P_k$ satisfy the transformation properties given in the following proposition:
\begin{prop}
The $P_k$ satisfy $P_k(z+1,\tau)=P_k(z,\tau)$ and
\be\label{Pkelliptic}
P_1(z+\tau,\tau) = P_1(z,\tau) +2\pi i \qquad P_k(z+\tau,\tau) = P_k(z,\tau) \qquad k>1\ .
\ee
Moreover, given $\begin{pmatrix}a&b\\c&d\end{pmatrix}\in SL(2,\mathbb{Z}),$
\begin{equation}
\tilde P_1(\gamma z, \gamma\tau)=(c\tau+d)\tilde P_1(z, \tau)-2\pi i cz,
\end{equation}
\begin{equation}
P_2(\gamma z, \gamma\tau)=(c\tau+d)^2P_2(z,\tau)-2\pi i c(c\tau+d),
\end{equation}
and if $k> 2,$
\begin{equation}P_k(\gamma z,\gamma\tau)=(c\tau+d)^kP_k(z,\tau).\end{equation}
\end{prop}
From this it follows that
\begin{cor}
 The $P_k$ have the following structure:
\begin{itemize}
\item For $k>2$, $P_k(z,\tau)$ is a Jacobi form of index 0 and weight $k$.
\item $P_2(z,\tau)$ is a quasi-Jacobi form of index 0, weight 2 and depth $(0,1)$.
\item $\tilde P_1(z,\tau)$ is a quasi-Jacobi form of index 0, weight 1 and depth $(1,0)$.
\end{itemize}
\end{cor}
For instance, for $P_2(z,\tau)$ we have $S_{0,1}(\tau,z)= -2\pi i$ and $T_1=0$. For $\tilde P_1(z,\tau)$ we have $S_{1,0}=-2\pi i$ and $T_1=2\pi i$. 

Next, as in equation (A.5) and (A.7) of \cite{Gaberdiel:2012yb}, for $j>0$ and $i\ge 0$ we define the functions
\be\label{definition_g_j^i}
g_j^i(z, \tau):=\frac{(2\pi i )^j}{(j-1)!}\sum_{n\in \Z\setminus\{0\}}n^{j-i-1}z^n\partial_\tau^i(1-q^n)^{-1}\ .
\ee
Thus for $j>i$
\be
g^i_{m+i}(z,\tau) = (2\pi i)^i \frac{(m-1)!}{(m+i-1)!} \partial^i_\tau P_{m}(z,\tau) \qquad i \geq 0\ .
\ee
Using the fact that $\partial_\tau$ maps  quasi-Jacobi forms of weight $k$ to quasi-Jacobi forms with weight $k+2$, it follows immediately that
for $i\geq 1$, $g^i_{m+i}(z,\tau)$ is a quasi-Jacobi form of index 0 and weight $m+2i$. More generally, $g^i_j(z,\tau)$ is a quasi-Jacobi form of index 0 and weight $i+j$. We give more details about this in appendix~\ref{app:gij}.

Finally, for computing explicit expressions, it can be useful to introduce the modular anomaly $\Delta f$: if $f(\tau,\underline{z})$ is a quasi-Jacobi form of weight $k$, then we define
\be
\Delta f = (c\tau+d)^{-k} f(\gamma \tau,\gamma\underline z) - f(\tau,\underline z).
\ee
Note that for $f$ of weight $k$ and $g$ of weight $l$
\be
\Delta (fg) = f \Delta g + (\Delta f) g + (\Delta f) (\Delta g)\ .
\ee
We have
\be
\Delta \tilde P_1 = - \frac{2\pi i c z}{c\tau+d}\ , 
\qquad \Delta P_2 = - \frac{2\pi ic}{c \tau+d},
\ee

\be
\Delta g_2^1 = \frac{2\pi i c}{c\tau+d}\left(P_1(z, \tau)+\pi i\right), 
\qquad \Delta g_3^1 =  \frac{2\pi ic}{c \tau+d}P_2(z,\tau)-\frac{1}{2}\left(\frac{2\pi i c}{c\tau+d}\right)^2
\ee
and if $m>2$
\be
\Delta g^1_{m+1} = \frac{2\pi i c}{c\tau+d} P_m(z,\tau)\ .
\ee

\subsection{VOA Correlators and Zhu's result}
Given an operator and states $a^i$ in some VOA, we define 
\be
F_M(O;(a^1,\zeta_1),\ldots,(a^n,\zeta_n);\tau) := \Tr_M O Y(\zeta_1^{L_0}a^1,\zeta_1) \ldots Y(\zeta_n^{L_0} a^n,\zeta_n) q^{L_0}.
\ee
Closely related we define 
\be
S_M(O;(a^1,z_1),\ldots,(a^n,z_n);\tau) := F_M(O;(a^1,e^{2\pi iz_1}),\ldots,(a^n,e^{2\pi iz_n});\tau)q^{-c/24}.
\ee
In working with vertex operators, there are different conventions for modes. The standard definition in mathematics is to use modes $a(n)$ as in
\be
Y(a,\zeta) = \sum_{n\in \Z} a(n)\zeta^{-n-1}\ .
\ee
In some cases the physics convention of using modes $a_n$ can be more useful,
\be\label{physmode}
Y(\zeta^{L_0}a,\zeta)= \sum_{n\in \Z} a_n \zeta^{-n}\ .
\ee
This is particularly because for a homogeneous state $a$ of weight $h_a$ with respect to $L_0$, $a_n=a(n+h_a-1)$ is homogeneous of degree $-n$. In particular, $a_0 =o(a):=a(h_a-1)$. Finally, when working on genus 1 surfaces, it can be useful to use torus modes $a[m]$ as introduced in \cite{Z} and defined by

\be Y[a,z]=\sum_{n\in \mathbb{Z}}a[n]z^{-n-1}:=Y(e^{2\pi i zL_0}a, e^{2\pi i z}-1).
\ee 

We have
\be
a[s]= s!(2\pi i)^{-s-1}\sum_{i=s}^\infty c(h_a,i,s)a(i) \ ,
\ee
where $h_a$ denotes the weight of $a$ with respect to $L_0$, and the coefficients $c(h_a,i,s)$ are defined by
\be
\binom{h_a-1+k}{i}=\sum_{s=0}^i c(h_a,i,s)k^s\ ,
\ee
so that in particular by lemma~4.3.1 in \cite{Z} 
\be
\sum_{i=s}^\infty c(h_a,i,s)z^i = \frac1{s!}(\ln(1+z))^s(1+z)^{h_a-1}\ .
\ee
These torus modes then lead to various identities that will be useful in what follows. For instance, for $b\geq 0$ we then have the identity 
\begin{multline}
\sum_{i=0}^\infty \sum_{k\neq 0}\binom{h_a-1+k}{i} k^b\frac{1}{1-q^k}x^k a(i)=
\sum_{i=0}^\infty \sum_{k\neq 0} \sum_{s=0}^i c(h_a,i,s)k^{b+s} \frac{1}{1-q^k}x^k a(i) =\\ 
\sum_{s=0}^\infty \left(\sum_{k\neq 0} k^{b+s} x^k\frac1{1-q^k}\right)\left(\sum_{i=s}^\infty c(h_a,i,s)a(i)\right)
= \sum_{s=0}^\infty \frac{(s+b)!}{(2\pi i)^{s+b+1}}P_{s+b+1}(x,q)\frac{(2\pi i)^{s+1}}{s!} a[s]\\
= \sum_{s=0}^\infty \frac{(2\pi i)^{-b}(s+b)!}{s!} P_{s+b+1}(x,q)a[s]\ .
\end{multline}
Similarly there is equation (B.21) in \cite{Gaberdiel:2012yb}:
\begin{equation}\label{B.21}
\sum_{t=0}^\infty\sum_{k\ne 0}\binom{h_a-1+k}{t}k^{-i}\partial_\tau^i\frac{1}{1-q^k}x^k a_{t-h_a+1}=\sum_{m=0}^\infty g_{m+1}^i(x,q)a[m].
\end{equation} In general we will denote by $h_a$ the weight of $a$ with respect to $L_0$, and by $[h_a]$ the weight with respect to $L[0]$.

Finally we recall the main result of Zhu, which is theorem 5.3.2 in \cite{Z} combined with the comment below it:
\begin{theorem}\label{thm:Zhu}
Let $V$ be a $C_2$-cofinite rational vertex operator algebra, let $M_1,\cdots, M_m$ denote the complete list of irreducible modules of $V$. Let $a^1,\cdots, a^n$ be homogeneous for $L[0]$ with weights $[h_{a^1}], \cdots, [h_{a^n}].$ Then for every $\gamma=\begin{pmatrix}a&b\\c&d\end{pmatrix}\in SL(2,\Z),$
\begin{multline*}
S_{M_i}((a^1, \frac{z_1}{c\tau+d}), \cdots, (a^n,\frac{z_n}{c\tau+d}), \frac{a\tau+b}{c\tau+d})=(c\tau+d)^{\sum_{k=1}^n [h_{a^k}]}\sum_{j=1}^mA_{\gamma, j}^iS_{M_j}((a^1, z_1), \cdots, (a^n, z_n), \tau)\ ,
\end{multline*}
where the $A_{\gamma, j}^i$ are constants depending only on $\gamma, i,j$.
\end{theorem}

For simplicity in this article we will take our VOAs $V$ to be holomorphic with central charge a multiple of 24, $c=24p$; holomorphic here means $C_2$-cofinite and rational with only one irreducible module, namely $V$ itself. In the language of Jacobi forms, Zhu's theorem \ref{thm:Zhu} then states that $S((a^1,z_1),\ldots, (a^n,z_n);\tau)$ is a weakly holomorphic Jacobi form of index 0 and weight $\sum_{k=1}^n [h_{a^k}]$ with a pole of order at most $p$ at the cusp $\tau=i\infty$.

\section{Correlators with zero modes}

\subsection{Recursion relation for mixed correlators}
Let us now derive a recursion relation for correlators with zero modes inserted.
First we introduce some notation. We will be working with index vectors $\vec s = (s_1,s_2,\ldots , s_{s})$ of varying length $s$. We write $\vec s =\emptyset$ for the vector of length $s=0$. We define
\be
\vec s \cup \vec t = (s_1,\ldots,s_s,t_1,\ldots,t_t)\ .
\ee
Moreover we write $\vec t \subset \vec s $ if $\vec t$ is a subtuple of $\vec s$, that is
\be
\vec t = (s_{i_1},s_{i_2},\ldots, s_{i_t}) \qquad 1 \leq i_1 < i_2 <\ldots < i_t \leq s\ .
\ee
If $\vec t$ is such a subtuple of $\vec s$, then we denote by $\vec s - \vec t$ the tuple of length $s-t$ obtained by removing the $i_j$-th entries from $\vec s$. Given a set of states $b^i$, for an index vector $\vec s$ define
\be
b_0^{\vec s} = \prod_{i}^\leftarrow b_0^{s_i}
\ee
where the arrow indicates the order of the factors so that

\be
b_0^{(s_1, \cdots, s_s)}=b_0^{s_s}b_0^{s_{s-1}}  \cdots b_0^{s_1}.
\ee Given a state $a$ we also define recursively
\be
d^{\emptyset}_k(a) = a_k\ , \qquad d^{s_i\cup \vec t}_k(a) = -\frac1{2\pi i}[b_0^{s_i},d^{\vec t}_k(a)]\ 
\ee for $s-i$-tuple $\vec t=(s_{i+1},\cdots, s_s).$
\begin{lem}\label{lem:dstate}
The $d_k^{\vec s}(a)$ are modes of the state \be
d^{\vec s}(a)= (-1)^{s} \prod_{i}^\rightarrow b^{s_i}[0] a\ .
\ee
\end{lem}
\begin{proof}We prove the lemma by inducting on $s$. The lemma is trivially true in the case that $s=0$, i.e. in the case that $\vec{s}=\emptyset.$ Suppose it is true for $\vec{s}$ such that $s=n\in \mathbb{N}.$ Then applying Equation (2.24) in \cite{Gaberdiel:2012yb} followed by the induction hypothesis, we have
\begin{multline*}
-\frac{1}{2\pi i}[b_0^{s_0}, Y(z^{L_0}d^{\vec{s}}(a),z)]=-\frac{1}{2\pi i}(2\pi i)Y(z^{L_0}b^{s_0}[0]d^{\vec{s}}(a), z)\\=-Y(z^{L_0}(-1)^{n} b^{s_0}[0]\prod_{s_i\in {\vec{s}}}^\rightarrow b^{s_i}[0] a,z)=(-1)^{n+1}Y(z^{L_0}\prod_{s_i\in s_0\cup {\vec{s}}}^\rightarrow b^{s_i}[0] a, z).
\end{multline*}
\end{proof}
We have 
\be
d^{\vec s}(d^{ \vec t }(a))= d^{\vec s \cup \vec t}(a) \ ,
\ee
where $\cup$ denotes the concatanation of vectors.

\begin{lem}\label{lem:comm}
    \be\label{comm}
[a_k, b_0^{\vec s}] = \sum_{\emptyset \neq \vec t\subset \vec s} (2\pi i)^{ t} b_0^{\vec s-\vec t} d_k^{\vec t}(a)
\ee
where the sum is over all non-empty subtuples ${\vec{t}}$ of ${\vec{s}}$.
\end{lem}
\begin{proof}
We induct on the length $s$ of $\vec s$. If $\vec s=(s_1)$ i.e. if $s=1$, we have
\be
[a_k,b_0^{\vec s}]=[d_k^\emptyset(a),b_0^{s_1}]=-[b_0^{s_1},d_k^\emptyset(a)]=2\pi i d_k^{\vec s}(a),
\ee which proves the base case.

Suppose Equation \eqref{comm} holds for ${s}=n$ where $n\ge 1.$ Then
\begin{eqnarray*}[a_k, b_0^{s_0\cup\vec{s}}]&=&[a_k, b_0^{\vec{s}}b_0^{s_0}]=[a_k, b_0^{\vec s}]b_0^{s_{0}}+b_0^{\vec s}[a_k, b_0^{s_{0}}]\\&=&\sum_{\emptyset\ne \vec{t}\subset \vec{s}}(2\pi i )^{t}b_0^{\vec{s}-\vec{t}}d_k^{\vec{t}}(a)b_0^{s_{0}}+b_0^{\vec s} 2\pi i d_k^{(s_{0})}(a)\\
&=&\sum_{\emptyset\ne{{\vec t}}\subset {\vec s}}(2\pi i)^{t}b_0^{{\vec s}-{\vec t}}([d_k^{\vec t}(a),b_0^{s_{0}}]+b_0^{s_{0}}d_k^{{\vec t}}(a))+2\pi i b_0^{{\vec s}}d_k^{(s_{0})}(a)\\&=&\sum_{\emptyset\ne {{\vec t}}\subset {\vec s}}(2\pi i)^{t}b_0^{{\vec s}-{\vec t}}(2\pi i d_k^{s_0\cup{{\vec t}}}(a)+b_0^{s_{0}}d_k^{\vec{t}}(a))+2\pi i b_0^{\vec{s}}d_k^{(s_{0})}(a)\\&{=}&\sum_{\emptyset\ne {{\vec t}}\subset {\vec s}}(2\pi i)^{t+1}b_0^{{\vec s}-{\vec t}}d_k^{s_0\cup{\vec t}}(a)+\sum_{\emptyset\ne {\vec t}\subset {\vec s}}(2\pi i)^{t}b_0^{s_0\cup{\vec s}-{\vec t}}d_k^{{\vec t}}(a)+2\pi ib_0^{\vec s}d_k^{(s_{0})}(a)\\&=&\sum_{{\vec t}\subset s_0\cup {\vec s}, s_{0}\in {\vec t}, {\vec{t}}\ne \{s_0\}}(2\pi i)^{t}b_0^{s_0\cup {\vec s}-{\vec t}}d_k^{{\vec t}}(a)+\sum_{\emptyset\ne {\vec t}\subset s_0\cup {\vec s},s_{0}\notin {\vec t}}(2\pi i)^{t}b_0^{s_0\cup {\vec s}-{\vec t}}d_k^{\vec t}(a)\\&&+2\pi ib_0^{{\vec s}}d_k^{(s_{0})}(a)\\&=&\sum_{\emptyset\ne {\vec t}\subset s_0\cup {\vec s}}(2\pi i)^{t}b_0^{s_0\cup {\vec s}-{\vec t}}d_k^{\vec t}(a).
\end{eqnarray*}
\end{proof}

For our main proposition let us introduce two definitions. For a vector $\vec{u}=(u_1, \cdots, u_t)$ we denote by $des(\vec u)$ the number of \emph{descents} of $\vec u$, that is the number of $j$, $1\le j\le t-1$,  such that $u_{j+1}<u_j$. Moreover the {\emph{Stirling numbers of the first kind}} $s(n,k)$ are defined as the coefficients in the expansion
\be (x)_n=x(x-1)\cdots (x-n+1)=\sum_{k=0}^ns(n,k)x^k\ .
\ee 
More details on these numbers and related combinatorial concepts can be found in Appendix \ref{Stirling_appendix}.
We are now ready to state our main recursion relation:
\begin{prop}\label{prop:rec}
Take $a^i$ and $b^j$ to be states, not necessarily primary. Then
\begin{multline*}
F(b_0^{(1,2,\cdots, r)}; (a^1, \zeta_1), \cdots, (a^n, \zeta_n); \tau)=F(b_0^{(1,2,\cdots, r)}a_0^1; (a^2, \zeta_2), \cdots, (a^n, \zeta_n);\tau)\\+\sum_{j=2}^n\sum_{\vec{s}\subsetneq (1,2,\cdots, r)}\sum_{\vec u}\sum_{i=0}^{u-des({\vec u})-1}\binom{u-des(\vec{u})-1}{i}\frac{1}{(i+des(\vec{u})+1)!}\sum_{t=1}^{i+des(\vec{u})+1}(2\pi i)^{u-t}s(i+des(\vec{u})+1, t)\\
\times\sum_{m=0}^\infty g_{m+1}^t(\frac{\zeta_j}{\zeta_1}, q)F(b_0^{\vec{s}}; (a^2, \zeta_2), \cdots, (d^{\vec{u}}(a^1)[m]a^j, \zeta_j), \cdots, (a^n, \zeta_n); \tau)\\
+\sum_{j=2}^n\sum_{m=0}^\infty g_{m+1}^0(\frac{\zeta_j}{\zeta_1}, q)F(b_0^{(1,2,\cdots,r)};(a^2, \zeta_2), \cdots, (a^1[m]a^j, \zeta_j),\cdots, (a^n, \zeta_n);\tau)
\end{multline*}
where the $\vec s$ sum is over all proper subtuples $\vec{s}$ of $(1,2,\cdots, r)$, including $\emptyset.$ The $\vec u$ sum is over all permutations $\vec{u}$ of $(1,2,\cdots, r)-{\vec{s}}$. $u$ denotes the length of $\vec{u}$, $des(\vec{u})$ denotes the number of descents of $\vec{u},$ and the $s(i+des(\vec{u})+1,t)$ are Stirling numbers of the first kind.
\end{prop}

\begin{proof}
The proof is very similar to the proof of the recursion relation in \cite{Gaberdiel:2012yb}, the main difference being the combinatorics of the terms. The basic idea is to expand the first vertex operator as in (\ref{physmode}). For the $k$-th term we then use the following identity for $k\ne 0$:
\begin{multline}
\zeta_1^{-k}F(b_0^{(1,2,\cdots,r)}a_k^1; (a^2, \zeta_2), \cdots, (a^n, \zeta_n);\tau)=\\\sum_{j=2}^n\left(\frac{\zeta_j}{\zeta_1}\right)^k\sum_{\vec s\subset (1, 2,\cdots, r)}\sum_{\vec u}(2\pi i)^u\frac{1}{1-q^k}C_{\vec u}\sum_{m=0}^\infty\binom{h-1+k}{m}F(b_0^{\vec s};(a^2, \zeta_2), \cdots, (d^{\vec u}_{m-h+1}(a^1)a^j, \zeta_j),\cdots, (a^n, \zeta_n);\tau)\ ,
\end{multline}
where the $C_{\vec u}$ are defined in definition~\ref{Ct}.
This identity is the analogue of equation (B.16) in \cite{Gaberdiel:2012yb}, and we prove it in lemma \ref{genLem} in the appendix. Next we use the expression in proposition \ref{Cprop} for the $C_{\vec{u}}$ and apply proposition \ref{diffop}. For $k\ne 0$ this gives
\begin{multline*}\zeta_1^{-k}F(b_0^{(1,2,\cdots, r)}a_k^1; (a^2, \zeta_2), \cdots, (a^n, \zeta_n);\tau)\\
= \sum_{j=2}^n\left(\frac{\zeta_j}{\zeta_1}\right)^k\sum_{\vec{s}\subsetneq (1,2,\cdots, r)}\sum_{\vec{u}}\sum_{i=0}^{u-des(\vec{u})-1}\binom{u-des(\vec{u})-1}{i}\frac{1}{(i+des(\vec{u})+1)!}\\\times\sum_{t=0}^{i+des(\vec{u})+1}(2\pi i)^{u-t}k^{-t}s(i+des(\vec{u})+1, t)\partial_\tau^t(1-q^k)^{-1}\\\times\sum_{m=0}^\infty\binom{h-1+k}{m}F(b_0^{\vec{s}};(a^2, \zeta_2), \cdots, (d_{m-h+1}^{\vec{u}}(a^1)a^j, \zeta_j),\cdots, (a^n, \zeta_n);\tau)\\+\sum_{j=2}^n\left(\frac{\zeta_j}{\zeta_1}\right)^k\frac{1}{1-q^k}\sum_{m=0}^\infty\binom{h-1+k}{m}F(b_0^{(1,2,\cdots, r)};(a^2, \zeta_2),\cdots, (a^1_{m-h+1}a^j, \zeta_j), \cdots, (a^n, \zeta_n);\tau).
\end{multline*}
We use this expression to sum $\zeta_1^{-k}F(b_0^{(1,2,\cdots, r)}a_k^1; (a^2, \zeta_2), \cdots, (a^n, \zeta_n);\tau)$ over all $k\in \mathbb{Z}$ and then apply Equation \eqref{B.21}. This gives us the result, except with the sum over $t$ from $0$ to $i+des(\vec{u})+1$ instead of from $1$ to $i+des(\vec{u})+1$. The result then follows from observing that $s(i+des(\vec{u})+1, 0)=0$ for all $\vec{u}$.
\end{proof}
This leads to our main theorem for the structure of mixed correlators:
\begin{thm}\label{main}
Let $V$ be a holomorphic VOA of central charge $24p$ and $a^i$, $1\le i\le n,$ and $b^j$, $1\le j\le r$, be $L[0]$-homogeneous states. The mixed correlator $S(b_0^{(1,2,\cdots,r)};(a^1,z_1),\ldots (a^n,z_n);\tau)$ is a weakly holomorphic quasi-Jacobi form  with a pole of at most order $p$ at $\tau= i\infty$. 
\end{thm}
\begin{proof}
We use induction in the number $r$ of zero modes. The base case $r=0$ follows from Zhu's theorem --- see the comment below theorem~\ref{thm:Zhu}. For the induction step we use proposition~\ref{prop:rec}. This allows us to express the correlator  with $r+1$ zero modes in terms of correlators with $r$ zero modes or less. We then simply observe that the functions $g^i_k(z,\tau)$ that appear in the process are all quasi-Jacobi forms that have no pole at $\tau = i\infty$.
\end{proof}
We note that in general the resulting quasi-Jacobi form will not have homogeneous weight. We will discuss in section~\ref{s:sVOA} under what assumptions on the $a^i$ and $b^j$ we obtain a homogeneous form. For the moment we simply observe that the weights of the form are at most $\sum [h_{a^i}] +\sum [h_{b^j}]$. Here the maximal weight comes from the term $t=u$ in the sum. See the proof of proposition~\ref{prop:HHA} for a proof of that statement.

By observing that quasi-Jacobi forms that are independent of all elliptic variables $z_i$ are quasi-modular forms, this gives an immediate corollary for zero-mode correlators:
\begin{cor}
Zero mode correlators are quasi-modular forms.
\end{cor}

\subsection{Commuting zero mode recursion relations}
In the previous section we gave a general structure theorem for mixed and zero-mode correlators of holomorphic VOAs.
Let us now discuss a special case for the recursion relation in proposition~\ref{prop:rec}: we assume that all zero modes $b_0^i$ commute.

Without loss of generality we can assume that all indices are distinct. We can therefore replace index tuples $\vec s$ by index sets $S$ and use the standard set notations $S-T$ and $T\cup T'$. The analogue of proposition~\ref{prop:rec} is then:

\begin{prop}\label{prop:reccomm}
Take states $a^i$ and $b^j$, not necessarily primary. Moreover assume that all $b_0^i$ commute. Then
\begin{multline}\label{reccomm}
F(b^{\{1,2,\ldots,r\}}_0;(a^1,\zeta_1),\ldots (a^n,\zeta_n);\tau)=F(b^{\{1,2,\ldots,r\}}_0a^1_0;(a^2,\zeta_2),\ldots (a^n,\zeta_n);\tau)
\\ + \sum_{j=2}^n \sum_{S\subset {\{1,2,\ldots,r\}}} \sum_{m\in \N} g^{|S|}_{m+1}(\zeta_j/\zeta_1,q) F(b_0^{{\{1,2,\ldots,r\}}-S};(a^2,\zeta_2), \ldots (d^S(a^1)[m]a^j,\zeta_j), \ldots (a^n,\zeta_n);\tau)
\end{multline}
where the middle sum is over all subsets $S$ of ${\{1,2,\ldots,r\}}$.
\end{prop}
\begin{proof}
We apply proposition \ref{prop:rec}, taking the $b_0^i$s to commute. In particular, we replace ordered tuples with sets. Let $R=\{1, \cdots, r\}$. Because the zero modes commute, we can simply replace a permutation $\vec u$ by 
the set $R-S$. The only thing to keep track of is how many descents $des(\vec u)$ the permutation introduced. We denote that number by $D$, where $D$ of course can take values from 0 to $|R|-|S|-1$. The number of permutations of the set $R-S$ that have $D$ descents is then given by the Eulerian number $A(|R|-|S|,D)$ --- see appendix~\ref{Stirling_appendix} for more details.
Proposition \ref{prop:rec} thus becomes
\begin{multline*}F(b_0^{R}; (a^1, \zeta_1), \cdots, (a^n, \zeta_n); \tau)=F(b_0^{R}a_0^1; (a^2, \zeta_2), \cdots, (a^n, \zeta_n);\tau)\\+\sum_{j=2}^n\sum_{S\subsetneq R}\sum_{D=0}^{|R|-|S|-1}A(|R|-|S|,D)\sum_{i=0}^{|R|-|S|-D-1}\binom{|R|-|S|-D-1}{i}\frac{1}{(i+D+1)!}\sum_{t=1}^{i+D+1}(2\pi i)^{|R|-|S|-t}s(i+D+1, t)\\\times\sum_{m\in \N} g_{m+1}^t(\frac{\zeta_j}{\zeta_1}, q)F(b_0^S; (a^2, z_2), \cdots, (d^{R-S}(a^1)[m]a^j, z_j), \cdots, (a^n, \zeta_n); \tau)\\+\sum_{j=2}^n\sum_{m\in \N} g_{m+1}^0(\frac{\zeta_j}{\zeta_1}, q)F(b_0^{R};(a^2, \zeta_2), \cdots, (a^1[m]a^j, \zeta_j),\cdots, (a^n, \zeta_n);\tau).
\end{multline*} 
Let us now reorganize the sums in the second line of the above equation. Because $s(i+D+1,t)=0$ if $t>i+D+1$, we may replace the sum over $t$ with a sum ranging from $t=0$ to $t=|R|-|S|$ and then pull the summation all the way to the left. Also, $s(i+D+1,t)=0$ if $i+D+1<t,$ and the binomial coefficient $\binom{|R|-|S|-D-1}{i}= 0$ if $i<0$, so the sum over $i$ can be replaced by a sum ranging from $i=t-D-1$ to $i=|R|-|S|-D-1$. We therefore have that the second term on the right hand side of the above expression is equal to 
\begin{multline*}
\sum_{j=2}^n\sum_{S\subsetneq R}\sum_{t=0}^{|R|-|S|}(2\pi i)^{|R|-|S|-t}\left(
\sum_{D=0}^{|R|-|S|-1}A(|R|-|S|,D)\sum_{i=t-D-1}^{|R|-|S|-D-1}\binom{|R|-|S|-D-1}{i}\frac{1}{(i+D+1)!}s(i+D+1, t)\right)\\
\times\sum_{m\in \N} g_{m+1}^t(\frac{\zeta_j}{\zeta_1}, q)F(b_0^S; (a^2, \zeta_2), \cdots, (d^{R-S}(a^1)[m]a^j, \zeta_j), \cdots, (a^n, \zeta_n); \tau).
\end{multline*}
Proposition \ref{identity_comm} then tells us that the sum in the large parentheses is actually equal to the Kronecker delta $\delta_{|R|-|S|,t}$. This allows us to evaluate the sum over $t$ to obtain
\begin{multline*}
F(b_0^{R}a_0^1; (a^2, \zeta_2), \cdots, (a^n, \zeta_n);\tau)\\+\sum_{j=2}^n\sum_{S\subsetneq R}\sum_{m\in \N} g_{m+1}^{|R|-|S|}(\frac{\zeta_j}{\zeta_1}, q)F(b_0^S; (a^2, z_2), \cdots, (d^{R-S}(a^1)[m]a^j, z_j), \cdots, (a^n, \zeta_n); \tau)\\+\sum_{j=2}^n\sum_{m\in \N} g_{m+1}^0(\frac{\zeta_j}{\zeta_1}, q)F(b_0^{R};(a^2, \zeta_2), \cdots, (a^1[m]a^j, \zeta_j),\cdots, (a^n, \zeta_n);\tau).
\end{multline*}
Combining the second two terms on the right hand side, and then interchanging the roles of $R-S$ and $S$ as they both range over all subsets of $R$, proves the result. 
\end{proof}

As an immediate corollary to proposition \ref{prop:reccomm} we obtain the recursion relation given in equation (2.27) in \cite{Gaberdiel:2012yb}:
Take $a^i$ and $b$ to be states, not necessarily primary. Then
\begin{multline}\label{Gaberdiel_recursion}
F((b_0)^r;(a^1, \zeta_1),\cdots, (a^n, \zeta_n);\tau)=F((b_0)^ra_0^1; (a^2, \zeta_2), \cdots, (a^n, \zeta_n);\tau)\\
+\sum_{s=0}^r\sum_{j=2}^n\sum_{m\in \mathbb{N}}\binom{r}{s}g_{m+1}^s\left(\frac{\zeta_j}{\zeta_1},q\right)F((b_0)^{r-s};(a^2, \zeta_2),\cdots, (d^{s}[m]a^j,\zeta_j),\cdots,(a^n,\zeta_n);\tau),
\end{multline}
where $(b_0)^r$ here denotes the $r$th power of $b_0$ and $d^{s}[m]$ is the $m$th square bracket mode of $d^{s}:=~(-1)^s(b[0])^sa^1$. This follows from the fact that there are $\binom{r}{s}$ subsets of length $s$ in $R$.

\section{Heavy Heisenberg Algebras}\label{s:sVOA}
By the results of \cite{MR1738180}, the generating function of the zero mode correlators of a Heisenberg field $a$ is a Jacobi form. We want to investigate if this type of result can be generalized to higher weight fields. In a first step, our goal is to mimic the pertinent properties of a Heisenberg algebra to higher weight fields. This leads us to the following definition:

\begin{defn}
Let $V$ be a holomorphic VOA and $A = \{a^1,a^2,\ldots a^N\}\subset V$ be a finite set of states homogeneous with respect to the $L[0]$ grading, not necessarily primary. Let $H_A:=span( L[-1]^k a^j: k\geq 0, j=1,\ldots N )$. We say that $A$ generates a \emph{heavy Heisenberg algebra} (HHA) $H_A$ if the zero modes of all states in $H_A$ commute,
\be\label{HHAcomm}
[b_0,c_0]= 0 \qquad \forall b,c \in H_A\ ,
\ee
and $H_A$ is closed under all square mode multiplications with $m\geq 0$, that is
\be\label{HHAact}
c[m]b \in H_A \qquad \forall m \geq 0,\ c, b \in H_A\ .
\ee
\end{defn}
By a slight abuse of terminology we will often refer to $A$ itself as a heavy Heisenberg algebra.

The idea here is that the first condition mimics the commutation condition for a Heisenberg algebra of multiple fields. The second condition mimics the closure condition of the Heisenberg algebra, that is the condition that no new fields other than the vacuum appear in the commutator of its modes. 

Let us now derive some results for zero mode correlators of a HHA and show that they are indeed of the desired form. We introduce the following lemma, which can be useful in computations:
\begin{lem}\label{lem:Lm1rels}
\be
a[m](L[-1])^n b = \sum_{k=0}^n \binom{m}{k} k!\binom{n}{k}L[-1]^{n-k} a[m-k] b 
\ee
\end{lem}
\begin{proof}
Induction in $n$. Base case $n=0$ clear.
\begin{multline}
a[m](L[-1])^{n+1}b= L[-1]a[m]L[-1]^nb + ma[m-1]L[-1]^nb\\
= \sum_{k=0}^n \binom{m}{k} k!\binom{n}{k}L[-1]^{n+1-k} a[m-k] b +
\sum_{k=0}^n m \binom{m-1}{k} k!\binom{n}{k}L[-1]^{n-k} a[m-1-k] b \\
= \sum_{k=0}^n \binom{m}{k} k!\binom{n}{k}L[-1]^{n+1-k} a[m-k] b +
\sum_{k=1}^{n+1} m \binom{m-1}{k-1} (k-1)!\binom{n}{k-1}L[-1]^{n+1-k} a[m-k] b \\
= \sum_{k=0}^n \binom{m}{k} k!\binom{n}{k}L[-1]^{n+1-k} a[m-k] b +
\sum_{k=1}^{n+1}  \binom{m}{k} k!\binom{n}{k-1}L[-1]^{n+1-k} a[m-k] b \\
= \sum_{k=0}^{n+1} \binom{m}{k} k!\binom{n+1}{k}L[-1]^{n+1-k} a[m-k] b,
\end{multline}
where in the last line we used the Pascal triangle recurrence relation for binomial coefficients.
\end{proof}
We use this to prove the following proposition:
\begin{prop}\label{prop:HHAgen}
Assume $A=\{a^1,\ldots, a^N\}\subset V$ satisfies
\be\label{CSAcomm}
[a^i_0,a^j_0]= 0 \qquad \forall i,j
\ee
and for all $m\geq 0$,
\be\label{CSAact}
a^i[m]a^j = \sum_{l=1}^N d_{ijl} (L[-1])^{k_{ijl}} a^l
\ee
for some constants $d_{ijl}$ and $k_{ijl}$, where the $k_{ijl}$ are completely fixed by homogeneity. Then $A$ generates a heavy Heisenberg algebra.
\end{prop}

\begin{proof}
First note that $L[-1]b$ has vanishing zero mode for any $b$, $(L[-1]b)_0=o(L[-1]b)=0$. The commutation condition (\ref{HHAcomm}) is thus automatically satisfied for any state $L[-1]^m a^i$. 

Next consider $c[l]b$ for $c=(L[-1])^la^i$ and $b=(L[-1])^n a^j$. First we can use lemma~\ref{lem:Lm1rels} to write
\be
c[m]b = \sum_{k=0}^{n} \binom{m}{k} k!\binom{n}{k}L[-1]^{n-k} c[m-k] a^j\ .
\ee
Note that by the definition of the binomial coefficients, the sum vanishes if $k>m$. It is thus a linear combination of terms of the form $L[-1]^r c[s]a^j$, $s\geq 0$. Next we use the fact that
\be
c[s]=(L[-1]^la^i)[s] = (-1)^l s(s-1)\ldots (s-l+1) a^i[s-l]\ .
\ee
Since this implies that terms with $l>s$ vanish, it follows that $c[m]b$ only contains terms of the form $L[-1]^r a^i[m]a^j$ with $m\geq 0$. (\ref{CSAact}) then implies that $c[m]b$ is indeed in $H_A$, satisfying (\ref{HHAact}).

\end{proof}

We will also need the following lemma, which is a straightforward generalization of proposition~4.3.1 in \cite{Z}:
\begin{lem}\label{lem:a0}
Assume that $a_0$ commutes with all $b_0^i$. Then
\be
\sum_{k=1}^n F(b_0^S;(a^1,\zeta_1),\ldots, (a[0]a^k,\zeta_k), \ldots, (a^n,\zeta_n)) = 0\ .
\ee
\end{lem}

\begin{prop}\label{prop:HHA}
The mixed correlators of a HHA are quasi-Jacobi forms of index 0 and homogeneous weight. Their weight is given by the sum of the $L[0]$-gradings of all insertions.
\end{prop}

\begin{proof}
We will prove this by induction in the number $r$ of zero mode insertions. The base case of $r=0$ is simply the original result of Zhu for ordinary correlators. For the case of $r+1$ we use proposition~\ref{prop:reccomm}. We simply need to check that the states appearing in the second line are still in the HHA. Note that $d^S(a^1)$ is of the form given in lemma~\ref{lem:dstate}. Since the HHA is closed under multiplications with $m\geq 0$, this is indeed the case. We can thus apply the induction assumption for the correlators with $b_0^{R-S}$ insertions together with the fact that the coefficients are given by functions $g^i_{m+1}$.

To establish the weight consider a term in the second line of (\ref{reccomm}) with $S\neq \emptyset$
.  Note that $d^S(a^1)$ has weight $\sum_{s\in S} h_{b^s} -|S|+h_{a^1}-m-1$ and $g^{|S|}_{m+1}$ has weight $|S|+m+1$. By using induction, the total weight of the term is thus 
\be
\sum_{j=2}^n h_{a^j}+\sum_{s\in S} h_{b^s} -|S|+h_{a^1}-m-1 + \sum_{s\in R-S} h_{b^s} +|S|+m+1
= \sum_{j=1}^n h_{a^j} + \sum_{s \in S}h_{b^s}\ .
\ee
For $S=\emptyset$ and $m>0$ the same argument goes through. For $S=\emptyset$ and $m=0$ lemma~\ref{lem:a0} shows that the inhomogeneous part of the $P_1$ cancel so that the above argument goes through for the $\tilde P_1$. 
\end{proof}

This implies the following result:

\begin{lem}\label{lem:FulltoZero}
Let $\{a^1,a^2,\ldots , a^N\}$ be a HHA and $\vec s \in N^n$, that is a $n$-tuple with entries in $\{1,\ldots, N\}$. Then
\be
F((a^{s_1},\zeta_1),(a^{s_2},\zeta_2),\ldots ,(a^{s_n},\zeta_n);\tau) = F(a_0^{\vec s};\tau) + \sum_{k=1}^{n-1}\sum_{\vec t \in N^k} G_{\vec t}(\zeta_1,\zeta_2,\ldots, \zeta_k;\tau) F(a_0^{\vec t};\tau)
\ee
where the $G_{\vec t}(\zeta_1,\ldots \zeta_n)$ are some functions in the algebra generated by the $g^l_{m+1}(\zeta_i/\zeta_j)$. 
\end{lem}
\begin{proof}
Apply recursion relation~\ref{prop:reccomm} until you are left with zero mode only correlators with no vertex operators inserted. By the definition of an HHA, the only states that appear in the process are of the form $L[-1]^ja^i$. However, since any states of the form $L[-1]a$ have vanishing zero modes, the only zero modes that can appear in the resulting zero mode correlators come from the original states $a^i$. The leading term comes from keeping track of the leading term in proposition~\ref{prop:reccomm}.
\end{proof}
Note that the sum over ordered $k$-tuples $\vec t$ can be replaced by a sum over unordered $k$-tuples since the zero modes commute.
Using the fact that unit triangular matrices over rings are invertible, we can invert the identities in the above lemma. Schematically, we can thus write
\be\label{ZerotoFull}
 F(a_0^{\vec s};\tau) = F((a^{s_1},\zeta_1),(a^{s_2},\zeta_2),\ldots (a^{s_n},\zeta_s);\tau) + \sum_{k=1}^{n-1}\sum_{\vec t \in N^k} \tilde G_{\vec t}(\zeta_1,\zeta_2,\ldots, \zeta_n;\tau) F( \{(a^{t_i},\zeta_i)\}_{i=1,\ldots k};\tau)
\ee
where  $F( \{(a^{t_i},\zeta_i)\}_{i=1,\ldots k};\tau)$ is to be understood as a full correlator with fields $(a^{t_i},\zeta_i)$ inserted for all $i=1,\ldots k$ and the $ \tilde G_{\vec t}$ are some functions generated by the $g^l_{m+1}(\zeta_i/\zeta_j)$. 

\section{Examples}\label{s:examples}
\subsection{Weight 1}
As a warm up, let us return to the known case of weight 1 HHA. This case was of course already considered in  \cite{MR1738180}. In particular, the modular transformation properties of the generating function of the zero mode correlators,
\be
\chi(\tau,z) 
= \sum_{n=0}^\infty \frac{(2\pi iz)^n}{n!} \Tr (a_0)^n q^{L_0-c/24}\ ,
\ee
were proven.
Their proof is fairly lengthy; see also \cite{Krauel:2012moc} for more details of it. We will use our recursion formula to give a shorter proof of the simplest case.

We consider the weight 1 HHA generated by $\{{\bf 1},a\}$, where $a$ is single field of weight 1. We can therefore use the specialized recursion relation (\ref{Gaberdiel_recursion}) here. Moreover we have $d^s(a)=0$ unless $s=0$. Recursion relation (\ref{Gaberdiel_recursion}) thus simplifies to
\begin{multline}\label{recMiyamoto}
F(a^s_0;(a,\zeta_1),\ldots ,(a,\zeta_n);\tau)\\
=F(a^{s+1}_0;(a,\zeta_2),\ldots, (a,\zeta_n);\tau)
+ \sum_{j=2}^n   g^{0}_{2}(\zeta_j/\zeta_1,q) F(a_0^s;(a,\zeta_2), \ldots, (d^0(a)[1]a,\zeta_j), \ldots, (a,\zeta_n);\tau)\\
=
F(a^{s+1}_0;(a,\zeta_2),\ldots, (a,\zeta_n);\tau)
+ \sum_{j=2}^n \langle a,a\rangle  \frac{P_2(\zeta_j/\zeta_1,\tau)}{(2\pi i)^2} F(a_0^s;(a,\zeta_2), \ldots, ({\bf{1}},\zeta_j), \ldots ,(a,\zeta_n);\tau),
\end{multline}where $\langle a, a\rangle$ is defined by the formula $\langle a , a \rangle {\bf 1}=a(1)a$.

The idea is to take a mixed correlator with $s$ zero mode insertions and $n$ full insertions and express it in terms of full correlators with at most $n+s$ insertions in the vein of (\ref{ZerotoFull}). We can then transform the full correlators and rewrite the result in terms of mixed mode correlators again. To do this, let us introduce a few definitions.

First, for $n\geq 0$ and $s\geq 0$, let $I_n= \{1,2,\ldots, n\}$ be a set of indices, and $I_{n,s}= I_n \cup \{n+1,\ldots, n+s\}$, with $I_0=\emptyset$. We call the indices in $I_n$ the \emph{full indices}, and the indices in $I_{n,s}-I_n$ the \emph{zero indices}. Let $U=\{u_1,\ldots, u_k\} \subset I_n$ be a subset of $I_n$. 
We can use these definitions to denote mixed correlators as
\be
F_s(I_n;\tau):= F(a^s_0;(a,\zeta_{1}),\ldots,(a,\zeta_{n});\tau)
\ee
and more generally
\be
F_s(U;\tau):= F(a^s_0;(a,\zeta_{u_1}),\ldots,(a,\zeta_{u_k});\tau)\ .
\ee
Next, we split up the indices in $I_{n,s}$ into configurations: we define a \emph{configuration} $c$ as the tuple $( \{ \{p^1_1,p^2_1\},\ldots, \{p^1_l,p^2_l\}\},U)$ of unordered pairs $\{p^1_j,p^2_j\}$ of indices in $I_{N,s}$, where at least one of $p^1_j, p^2_j$ is a zero index, and $U= I_{n,s}-\cup_j \{p^1_j,p^2_j\}$ is the set of the remaining unpaired indices in $I_{n,s}$. Note that $|U|=n+s-2l$. Let $\cC_{n,s}$ denote the set of all such configurations. We will now use these configurations to give an explicit expression for (\ref{ZerotoFull}):
\begin{lem}\label{lem:w11}
For $n\geq 0$ let $I_n= \{1,2,\ldots, n\}$ be a set of indices.

Then
\be
F_s(I_n;\tau) = \sum_{c\in \cC_{n,s}} F_0(U;\tau)\prod_{j=1}^l(- \langle a,a\rangle) \frac{P_2(\zeta_{p^1_j}/\zeta_{p^2_j},\tau)}{(2\pi i)^2}. 
\ee
where $U$ is the set of unpaired indices in the configuration $c$, and the $p^{1}_j,p^{2}_j$ run over the pairs in the configuration $c$.
\end{lem}
\begin{proof}
We prove this by induction in $s$. The base case $s=0$ has only one configuration with no pairs so that the identity is trivially true. For $s+1$ we write the recursion relation (\ref{recMiyamoto}) as
\begin{multline}
F_{s+1}(I_n;\tau) =F(a^{s+1}_0;(a,\zeta_1),\ldots ,(a,\zeta_n);\tau)
\\
=
F(a^s_0;(a,\zeta_1),\ldots ,(a,\zeta_n),(a,\zeta_{n+1});\tau)
- \sum_{j=1}^n \langle a,a\rangle  \frac{P_2(\zeta_j/\zeta_{n+1},\tau)}{(2\pi i)^2} F(a_0^s;(a,\zeta_1), \ldots ,({\bf{1}},\zeta_j), \ldots (a,\zeta_n);\tau)
\end{multline}
We then apply the induction hypothesis to the two terms. The first term gives all configurations in $\cC_{n,s+1}$ where $n+1$ is either unpaired or pairs with another zero index. The second term gives all configurations where $n+1$ pairs with a full index. In total we recover all configurations $\cC_{n,s+1}$.
\end{proof}
For $n=0$ this then gives an explicit expression for (\ref{ZerotoFull}).
To compare to the notation of \cite{MR1738180, Krauel:2012moc}, note that $\cC_{0,s}$ can be written as $I(s)$, the set of involution permutations of $s$ elements. The pairs are then given by $(j,\sigma(j))$ such that $j<\sigma(j)$ and $U$ is the set $f(\sigma)$ of fixed points of $\sigma$.

Next we use lemma~\ref{lem:w11} to find the modular transformation of the zero mode correlators;
this is the analogue of lemma 4.2 in \cite{MR1738180}:
\begin{lem}\label{lemMiyamoto}
\be
(c\tau+d)^{-s}F_s(\emptyset ; \gamma\tau) = \sum_{k=0}^{s/2} F_{s-2k}(\emptyset;\tau) \left(\frac{2\pi ic \langle a,a\rangle}{(c\tau+d)(2\pi i)^2}\right)^k \frac{s!}{2^k k! (s-2k)!}
\ee
\end{lem}
\begin{proof}
We have
\begin{multline}
F_s(\emptyset ; \gamma\tau) =
\sum_{c\in \cC_{0,s}} (c\tau+d)^{|U|}F_0(U;\tau)\prod_{j=1}^l\frac{- \langle a,a\rangle}{(2\pi i)^2}(c\tau+d)^2 \left(P_2(\zeta_{p^1_j}/\zeta_{p^2_j},\tau) - \frac{2\pi i c}{c\tau+d}\right)
\\
= (c\tau+d)^s\sum_{c\in \cC_{0,s}} F_0(U;\tau)\prod_{j=1}^l\frac{- \langle a,a\rangle}{(2\pi i)^2} \left(P_2(\zeta_{p^1_j}/\zeta_{p^2_j},\tau) - \frac{2\pi i c}{c\tau+d}\right).
\end{multline}
We next proceed by collecting terms that appear from expanding the parentheses in the product. First, the terms where we keep all the $P_2$ give back $(c\tau+d)^s F_s(\emptyset; \tau)$. Next we collect all the terms where instead of $P_2(\zeta_1/\zeta_2)$ we take the $-2\pi ic/(c\tau+d)$ term. We observe that these terms correspond to summing over configurations $\cC_{0,s-2}$. This means that the terms sum up to 
\be
(c\tau+d)^s F_{s-2}(\emptyset,\tau) \frac{2\pi ic \langle a,a\rangle}{(c\tau+d)(2\pi i)^2}\ .
\ee
The same holds for any other choice of $P_2(\zeta_i/\zeta_j)$. In total there are $s(s-1)/2$ such choices. Proceeding further with this approach, the terms omitting $k$ $P_2$s give a total contribution 
\be
(c\tau+d)^s F_{s-2k}(\emptyset,\tau) \left(\frac{2\pi ic \langle a,a\rangle}{(c\tau+d)(2\pi i)^2}\right)^k \frac{s!}{2^k k! (s-2k)!}\ .
\ee
The combinatorial factor comes from choosing $k$ unordered pairs out of $s$ indices; since the order of the pairs themselves doesn't matter we have an additional factor $k!$ in the denominator.
Summing over $k$ then establishes the claim.

\end{proof}
This gives us all the tools we need to prove the simplest case of the main proposition in \cite{MR1738180}:
\begin{prop}
\be
\chi(\gamma\tau, \gamma z) = \exp\left(\frac{\pi i \langle a,a\rangle c z^2}{c\tau+d}\right) \chi(\tau,z)\ .
\ee
\end{prop}
\begin{proof}
\begin{multline}
\chi(\gamma\tau, \gamma z)
=\sum_{m=0}^\infty \frac{1}{(2m)!}F_{2m}(\emptyset;\gamma\tau) \left( \frac{2\pi i z}{c\tau+d}\right)^{2m}
\\
=\sum_{m=0}^\infty \sum_{k=0}^m \frac{1}{(2m)!} (c\tau+d)^{2m}\left(\frac{2\pi ic \langle a,a\rangle}{(c\tau+d)(2\pi i)^2}\right)^k \frac{(2m)!}{2^k k!(2m-2k)!} F_{2m-2k}(\emptyset;\tau) 
 \left( \frac{2\pi i z}{c\tau+d}\right)^{2m}\\
 = \sum_{m=0}^\infty \sum_{k=0}^m \frac{1}{k!(2m-2k)!} \left(\frac{\pi ic \langle a,a\rangle z^2}{c\tau+d}\right)^k  F_{2m-2k}(\emptyset;\tau) 
 \left( 2\pi i z\right)^{2m-2k}\\
 = \sum_{k=0}^\infty \frac1{k!}\left(\frac{\pi ic \langle a,a\rangle  z^2}{c\tau+d}\right)^k \sum_{n=0}^\infty \frac1{(2n)!}F_{2n}(\emptyset;\tau) (2\pi iz)^{2n}\\
 = \exp\left(\frac{\pi ic \langle a,a\rangle z^2}{c\tau+d}\right) \chi(\tau,z)
\end{multline}
where in line 4 we introduced a new summation variable $n=m-k$. 
\end{proof}

\subsection{Weight 2: General structure}
For weight 2, one obvious example to choose is $v:=L[-2]{\bf 1}=\omega-\frac{c}{24}{\bf 1}$. In that case we have $v_0=L_0 -c/24$. Zero mode correlators then automatically have the right form because we can write
\be
\Tr (v_0)^n q^{L_0-c/24} = \left( q \frac{d}{dq}\right)^n \Tr q^{L_0-c/24}\ ,
\ee
which indeed transforms like a quasi-modular form of weight $2n$. Our results thus add nothing new in this case.

Instead we will thus consider a slightly more general example  inspired by the above.
Let $a(-1){\bf 1}\in V$ such that $\langle a(-1){\bf 1}\rangle $ is a subVOA of $V$ isomorphic to the Heisenberg vertex operator algebra. Define $x=a[-1]^2{\bf 1}.$ Then

\begin{equation}\label{torusmodes}
x[m]x=
    \begin{cases}
        \displaystyle\frac{4}{2\pi i}(a(-1)^2{\bf 1}+a(-1)a(-2){\bf 1})=\frac{2}{(2\pi i)^2}L[-1]x
        & \text{if } m=0\\
        \displaystyle\frac{4}{(2\pi i)^2}x & \text{if } m=1\\
         0 & \text{if }  m=2\\
        \displaystyle \frac{2}{(2\pi i)^4}{\bf 1} & \text{if } m=3\\
         0 & \text{if } m>3.\\
    \end{cases}
\end{equation}
$\{{\bf 1}, x\}$ is then an HHA.
Let us show how (\ref{ZerotoFull})
works in this case and give explicit expressions for $s=1,2,3$. For this we use the simpler case (\ref{Gaberdiel_recursion}) of the recursion relation. In the process we can use
\be
d^s(x) =(-1)^s\left(\frac{2}{(2\pi i)^2}\right)^s (L[-1])^s x
\ee
as follows from induction and the fact that $x[0]$ commutes with $L[-1]$, and 
\be
d^s(x)[m]= \left(\frac{2}{(2\pi i)^2}\right)^s m(m-1)\ldots (m-s+1) x[m-s]\ .
\ee
For $s=1$ we simply have
\begin{eqnarray*}
F(x_0;\tau)=F((x, \zeta_1);\tau)
\end{eqnarray*}
For $s=2$ we have
\begin{multline}\label{w2s2}
F(x_0^2;\tau)= F(x_0; (x, \zeta_2);\tau) = F((x, \zeta_1), (x, \zeta_2);\tau)-
\sum_{m\in \N} P_{m+1}(\frac{\zeta_2}{\zeta_1}, q)F((x[m]x, \zeta_2);\tau)\\
=F((x, \zeta_1), (x, \zeta_2);\tau)-\frac{4}{(2\pi i)^2}P_2(\frac{\zeta_2}{\zeta_1},q)F((x,\zeta_2);\tau)-\frac{2}{(2\pi i )^4}P_4(\frac{\zeta_2}{\zeta_1},q)F(({\bf{1}},\zeta_2);\tau)\ .
\end{multline}
Note that the vanishing of the $P_1$ term in the above computation is because the zero mode of $x[0]x=\frac{2}{(2\pi i)^2}L[-1]x$ is $0$. This is necessary to make the resulting expression homogeneous of weight 4.

Finally consider $s=3$. We have
\begin{multline}\label{Fx03}
  F(x_0^3; \tau) = F(x_0^2;  (x, \zeta_3);\tau)=  F(x_0, (x, \zeta_2), (x, \zeta_3);\tau)-\sum_{i=0}^1\sum_{m\in \mathbb{N}}g_{m+1}^i(\frac{\zeta_3}{\zeta_2},q)F(x_0^{1-i};(d^{(i)}[m]x,\zeta_3);\tau)\\
  =F(x_0, (x, \zeta_2), (x, \zeta_3);\tau)-\sum_{m=1}^3P_{m+1}(\frac{\zeta_3}{\zeta_2}, q)F(x_0; (x[m]x, \zeta_3);\tau)
  +\sum_{m=2}^4g_{m+1}^1(\frac{\zeta_3}{\zeta_2}, q)F( ((x[0]x)[m]x, \zeta_3);\tau)
\end{multline}
where we used lemma~\ref{lem:a0} to eliminate the $m=0$ in the first sum and (because of $(x[0]x)[m] \sim x[m-1]$) the $m=0,1$ terms in the second sum. Next we have 
\begin{multline}\label{Fxxxrec}
F(x_0; (x, \zeta_2), (x, \zeta_3);\tau) = F((x, \zeta_1), (x, \zeta_2), (x, \zeta_3);\tau)\\ - \sum_{m=0}^3 P_{m+1}(\frac{\zeta_2}{\zeta_1},q)F((x[m]x,\zeta_2),(x,\zeta_3);\tau) - \sum_{m=0}^3 P_{m+1}(\frac{\zeta_3}{\zeta_1},q)F((x,\zeta_2),(x[m]x,\zeta_3);\tau)
\end{multline}
Note that because of lemma~\ref{lem:a0} we can replace the $P_1$ by $\tilde P_1$ in the sums. Finally we have
\be
F(x_0; (x[m]x, \zeta_3);\tau)= F((x,\zeta_1),(x[m]x,\zeta_3);\tau)-\sum_{k=1}^\infty P_{k+1}(\frac{\zeta_3}{\zeta_1},q)F((x[k]x[m]x,\zeta_3);\tau)\ .
\ee
Let us now work out what this means for the modular transformations of the zero mode correlators. 
For $s=1$ we obviously have
\be
(c\tau+d)^{-2}F(x_0;\gamma \tau) = F(x_0;\tau)\ .
\ee
For $s=2$ we find
\be
(c\tau+d)^{-4}F(x_0^2;\gamma \tau) = F(x_0^2;\tau) +\frac{4 c}{2\pi i(c\tau+d)} F((x,\zeta_2);\tau)
= F(x_0^2;\tau) + \frac{4 c}{2\pi i(c\tau+d)} F(x_0;\tau)\ .
\ee
Finally consider $s=3$. We have
\begin{multline}
\Delta F(x_0; (x, \zeta_2), (x, \zeta_3);\tau) = \frac{2\pi i c}{(c\tau+d)} (F((x[1]x,\zeta_2),(x,\zeta_3);\tau) + F((x,\zeta_2),(x[1]x,\zeta_3);\tau)\\
= \frac{2\pi i c}{(c\tau+d)}\frac{8}{(2\pi i)^2} F((x,\zeta_2),(x,\zeta_3);\tau) 
\end{multline}
Next we have
\be
\Delta F(x_0; (x[m]x, \zeta_3);\tau) = \frac{2\pi ic}{c\tau+d}F((x[1]x[m]x,\zeta_3);\tau).
\ee
This means the modular anomaly of the first sum in the second line of (\ref{Fx03}) is
\begin{multline}
-\sum_{m=1}^3 P_{m+1}(\frac{\zeta_3}{\zeta_2},q)  \frac{2\pi ic}{c\tau+d}F((x[1]x[m]x,\zeta_3);\tau) + \frac{2\pi ic}{c\tau+d} F(x_0;(x[1]x,\zeta_3);\tau)+ \left(\frac{2\pi ic}{c\tau+d}\right)^2F((x[1]x[1]x,\zeta_3);\tau)\\
=-\frac{2\pi ic}{c\tau+d} \frac 4{(2\pi i)^2}P_{2}(\frac{\zeta_3}{\zeta_2},q)F((x[1]x,\zeta_3),\tau) + \frac{2\pi ic}{c\tau+d} F(x_0;(x[1]x,\zeta_3);\tau)+
\left(\frac{2\pi ic}{c\tau+d}\right)^2 \frac 4{(2\pi i)^2}F((x[1]x,\zeta_3);\tau)\\
=\frac{2\pi ic}{c\tau+d} \frac{4}{(2\pi i)^2}\left(-P_{2}(\frac{\zeta_3}{\zeta_2},q)F((x[1]x,\zeta_3),\tau) + F(x_0;(x,\zeta_3);\tau)\right)+
\left(\frac{2\pi ic}{c\tau+d}\right)^2 \frac 4{(2\pi i)^2}F((x[1]x,\zeta_3);\tau)\\=
\frac{2\pi ic}{c\tau+d} \frac{4}{(2\pi i)^2}\left(
F((x, \zeta_2), (x, \zeta_3);\tau)-
2P_2(\frac{\zeta_3}{\zeta_2},q)F((x[1]x,\zeta_3);\tau)-
P_4(\frac{\zeta_3}{\zeta_2},q)F((x[3]x,\zeta_3);\tau)\right)\\
+
\left(\frac{2\pi ic}{c\tau+d}\right)^2 \frac {16}{(2\pi i)^4}F((x,\zeta_3);\tau)
\end{multline}
and for the second sum in the second line of  (\ref{Fx03}) is
\begin{multline}
-\frac{1}{2}\left(\frac{2\pi i c}{c\tau+d}\right)^2F((x[0]x)[2]x, \zeta_3);\tau)+\sum_{m=2}^4\frac{2\pi ic}{c\tau+d} P_m(\frac{\zeta_3}{\zeta_2},q) F( ((x[0]x)[m]x, \zeta_3);\tau)\\
= -\frac{1}{2}\left(\frac{2\pi i c}{c\tau+d}\right)^2\left(\frac{2}{(2\pi i )^2}\left(-2F((x[1]x, \zeta_3);\tau)\right)\right)+\frac{2\pi ic}{c\tau+d}\frac{2}{(2\pi i)^2}
\sum_{m=2}^4 -m P_m(\frac{\zeta_3}{\zeta_2},q) F( (x[m-1]x, \zeta_3);\tau)\\
=\left(\frac{2\pi i c}{c\tau+d}\right)^2\frac{8}{(2\pi i )^4}F((x, \zeta_3);\tau)+\frac{2\pi ic}{c\tau+d}\frac{2}{(2\pi i)^2} \left(
-2P_2(\frac{\zeta_3}{\zeta_2},q) F( (x[1]x, \zeta_3);\tau) -4P_4(\frac{\zeta_3}{\zeta_2},q) F( (x[3]x, \zeta_3);\tau)
\right)
\end{multline}
Collecting everything we get
\begin{multline}
\Delta F(x_0^3;\tau) =
\frac{2\pi ic}{c\tau+d} \frac{4}{(2\pi i)^2}\left(3F((x,\zeta_2),(x,\zeta_3);\tau) - 3P_2(\frac{\zeta_3}{\zeta_2},q) F( (x[1]x, \zeta_3);\tau) - 3P_4(\frac{\zeta_3}{\zeta_2},q) F( (x[3]x, \zeta_3);\tau)
\right)
\\ + \left(\frac{2\pi ic}{c\tau+d}\right)^2 \frac {24}{(2\pi i)^4}F((x,\zeta_3);\tau)
\end{multline}
In total we thus get
\be
(c\tau+d)^{-6}F(x_0^3;\gamma \tau) = F(x_0^3;\tau)
+\frac{12 c}{2\pi i(c\tau+d)}F(x_0^2;\tau)+24 \left(\frac{c}{2\pi i(c\tau+d)}\right)^2 F(x_0;\tau)\ .
\ee
We see that the up to this order, the zero mode correlators have the right transformation structure; namely that their modular anomalies are again zero mode correlators. It is natural to conjecture that the same continues to hold for higher $s$, although we did not check this.

\subsection{Weight 2: Lattice VOAs}
Let us now discuss a specific realization of the weight 2 HHA described above in the context of lattice VOAs. This is essentially the case discussed in \cite{MR1827085}, and their results together with \cite{MR1877753,MR3248159} allow us to give explicit expressions for the zero mode correlators.

Let $L$ be an even positive definite lattice of dimension $l$ with bilinear form denoted by $\langle\cdot, \cdot\rangle$. We can associate a vertex operator algebra $V_L$ to it \cite{MR843307,MR996026}. As a vector space, the VOA is given by $V_L=S(\widehat{\mathfrak{h}}^{-})\otimes \mathbb{C}_\epsilon[L]$ where $S(\widehat{\mathfrak{h}}^{-})$ is the Heisenberg algebra of the ambient space $\mathfrak{h}=L\otimes_\mathbb{Z}\mathbb{C}$ and $\mathbb{C}_\epsilon[L]$ is the group algebra of $L$ twisted by a cocycle $\epsilon$ that satisfies $\epsilon(\alpha,\alpha)=(-1)^{\langle \alpha,\alpha \rangle}$ and $\epsilon(\alpha,\beta)/\epsilon(\beta,\alpha)=(-1)^{\langle\alpha,\beta\rangle}$. 
Given $\alpha\in L$, we denote the image of $\alpha$ in $\mathbb{C}_\epsilon[L]$ by $\mathfrak{e}_\alpha$. 
Given $a\in \mathfrak{h}$, we define its modes via
\begin{equation}
Y(a,z)=\sum_{n\in \mathbb{Z}}a(n)z^{-n-1}\ .
\end{equation}
Let $h_1, \cdots, h_l$ be an orthonormal basis for $\mathfrak{h}$. Then the conformal vector is given by
\begin{equation}
\omega=\sum_{i=1}^l\frac{1}{2}h_i(-1)^2\ .
\end{equation}
$V_L$ is then spanned by the vectors
\be
h_1(-n_{1k_{1}})\cdots h_i(-n_{ij})\cdots h_l(-n_{lk_{l}}) \mathfrak{e}_{\alpha}
\ee
where $\alpha\in L$, and $n_{ij}\in \mathbb{Z}_+$, and their conformal weights are given by 
\be
\langle \alpha,\alpha\rangle/2+ \sum_{i=1}^l\sum_{j=1}^{k_i} n_{ij}\ .
\ee

\begin{prop}\label{quasimod}Let $v=h_1[-1]^2{\bf 1}=h_1(-1)^2{\bf 1}-\frac{1}{12}{\bf 1} \in V_L$ and let $v_0$ denote the zero mode of $v$.  Fix $n\ge 0$. Then
\be
\Tr_{V_L}v_0^nq^{L_0-l/24}
=\sum_{j=0}^n\binom{n}{j}\sum_{\alpha\in L}\langle h_1, \alpha\rangle^{2j}q^{\langle \alpha, \alpha\rangle/2}\frac{1}{\eta(q)^{l-1}}\left(2q\frac{d}{dq}\right)^{n-j}\eta(q)^{-1}\ .
\ee
\end{prop}  
\begin{proof}Given 
\begin{equation}\label{vector}
h_1(-n_{1k_{1}})\cdots h_i(-n_{ij})\cdots h_l(-n_{lk_{l}}) \mathfrak{e}_{\alpha}\ ,
\end{equation} 
we have for the action of $v_0$
\begin{eqnarray}
\lefteqn{v_0^nh_1(-n_{1k_{1}})\cdots h_i(-n_{ij})\cdots h_l(-n_{lk_{l}}) \mathfrak{e}_{\alpha}}\nonumber\\
&=&(\langle h_1, \alpha\rangle^2+2\sum_{j=1}^{k_1} n_{1j}-\frac1{12})^nh_1(-n_{1k_{1}})\cdots h_i(-n_{ij})\cdots h_l(-n_{lk_{l}}) \mathfrak{e}_{\alpha}\nonumber\\
&=&\sum_{j=0}^n\binom{n}{j}\langle h_1, \alpha\rangle^{2j}\left(2\sum_{m=1}^{k_1} n_{1m}-\frac1{12}\right)^{n-j}h_1(-n_{1k_{1}})\cdots h_i(-n_{ij})\cdots h_l(-n_{lk_{l}}) \mathfrak{e}_{\alpha}\label{sum}.
\end{eqnarray}
As usual, when summing over all $n_{1m}$ without any modes inserted we obtain the generating function of integer partitions $\prod_{i=1}^\infty (1-q^i)^{-1}=q^{1/24}\eta(q)^{-1}$. When inserting the factor $(2\sum_m n_{1m} - 1/12)$ in this sum, we can replace it by a derivative $2q \frac{d}{dq}$. The sums over the $n_{jm}$ for $j>1$ simply gives the usual $q^{1/24}\eta(q)^{-1}$ factors.
Therefore,
\be
\Tr_{V_L}v_0^nq^{L_0-l/24}
=\sum_{j=0}^n\binom{n}{j}\sum_{\alpha\in L}\langle h_1, \alpha\rangle^{2j}q^{\langle \alpha, \alpha\rangle/2}
\frac{1}{\eta(q)^{l-1}}\left(2q\frac{d}{dq}\right)^{n-j}\eta(q)^{-1}\ 
\ee
as claimed.
\end{proof}
Let us now specialize to the case of even unimodular lattices whose dimension $l$ is a multiple of 24. In that case $V_L$ is a holomorphic VOA with central charge a multiple of 24. 
By the results of \cite{MR1877753} and \cite{MR3248159}, $\sum_{\alpha\in L}\langle h_1, \alpha\rangle^{2j}q^{\langle \alpha, \alpha\rangle/2}$ is a quasi-modular form of weight $l/2+2j$. On the other hand we know that $\frac{1}{\eta(q)^{l-1}}\left(2q\frac{d}{dq}\right)^{n-j}\eta(q)^{-1}$ is a quasi-modular form of weight $-l/2+2(n-j)$. In total the zero mode correlator is thus a quasi-modular form of weight $2n$, exactly as predicted by proposition~\ref{prop:HHA} specialized to zero mode correlators.

\subsection{Weight 3}
Finally let us make a very brief comment on the weight 3 case.
We will not work out a concrete example here, but want to point out that functions of the type $\Theta(x,\tau,z)$ introduced in  \cite{MR1363056} may be candidates for the generating functions of weight 3 HHA zero mode correlators.

\appendix

\section{Transformation properties of $g^i_j(z,\tau)$}\label{app:gij}

\begin{proposition}Given any $i,j\in \mathbb{N}$, with $j>0$, $g_j^i(z,\tau)$ is a quasi-Jacobi form of weight $i+j$.
\end{proposition} 
\begin{proof}
First we note that $\wp_k(z,\tau)$ $k\ge 1$ converge absolutely and uniformly on compact sets away from lattice points \cite{MR890960}. This means we can take derivatives and anti-derivatives, as long as we avoid introducing branch cuts.

Let us first establish (\ref{quasimodular}).
If $i<j$ we can use the identity 
\begin{equation}\label{0}g_j^i(z,\tau)=(2\pi i)^i\frac{(j-i-1)!}{(j-1)!}\partial_\tau^iP_{j-i}(z,\tau)\ .
\end{equation}
Since for all $k>0$, $P_k(z,\tau)$ is a quasi-Jacobi form of weight $k$, and application of $\partial_\tau$ sends quasi-Jacobi forms of weight $k$ to quasi-Jacobi forms of weight $k+2$, equation \eqref{0} therefore establishes that $g_j^i(z,\tau)$ is a quasi-Jacobi form of weight $i+j$. 

It remains to establish this for $i\ge j$. 
First observe that given any $i,j\in \mathbb{N}$ with $j>0$, 
\begin{equation}\label{0a}
\partial_\tau g_j^i(z,\tau)=\frac{j}{2\pi i}g_{j+1}^{i+1}(z,\tau)\ ,
\end{equation}
which implies that in order to prove  the result we need only establish that any $g_1^m(z,\tau)$ with $m\ge 1$ is a quasi-Jacobi form of weight $m+1$.

Let $m\in \mathbb{N}$, $m>0$ and define $q_z:=e^{2\pi i z}$. Combining the definition of $g_1^0(z,\tau)=P_1(z,\tau)$ with equations (3.6) and (3.11) in \cite{Z}, we have
\begin{equation}\label{1}
2\pi i \sum_{n\ne 0}q_z^n(1-q^n)^{-1}=-\frac{1}{z}+\sum_{n=0}^\infty G_{2n+2}(\tau)z^{2n+1}-\pi i,
\end{equation} and applying $\partial_\tau^m$ to both sides of this equation we have
\begin{equation}
2\pi i \sum_{n\ne 0}q_z^n\partial_\tau^m(1-q^n)^{-1}=\sum_{n=0}^\infty \partial_\tau^m G_{2n+2}(\tau)z^{2n+1}.
\end{equation}Regarding both sides of this last equation as formal series in $z$ we integrate both sides formally with respect to $z$, sending each $z^l$, $l\ne 0$, to $\frac{z^{l+1}}{l+1}$. In fact, because of the comments on convergence made above and the fact that both sides are regular at $z=0$, we can write the formal integral as an actual integral $\int_0^z dz$.
We perform this integration $m$ times, observing that if $n\ne 0$,
$ \int q_z^n\frac{1}{q_z}dq_z=\frac{1}{2\pi i}\int q_z^n dz$. We thus have
\begin{equation}\label{3}
g_1^m(z,\tau)=2\pi i \sum_{n\ne 0}n^{-m}q_z^n\partial_\tau^m(1-q^n)^{-1}=(2\pi i)^m\sum_{n=0}^\infty \frac{\partial_\tau^mG_{2n+2}(\tau)z^{2n+1+m}}{(2n+2)(2n+3)\cdots (2n+1+m)}.
\end{equation}
Since application of $\partial_\tau$ sends quasi-modular forms of weight $k$ to quasi-modular forms of weight $k+2$, each $\partial_\tau^mG_{2n+2}(\tau)$ is quasi-modular of weight $2(m+n+1)$ and depth $m$. Taking into account the factor of $z^{2n+1+m}$, each term in the sum thus transforms under modular transformations like a quasi-Jacobi form of weight $m+1$ and depth $m$.
The left hand side of equation \eqref{3} is equal to $g_1^m(z, \tau)$ so this establishes the modular transformation properties (\ref{quasimodular}) of $g_1^m(z, \tau)$ required in proving it is a quasi-Jacobi form of weight $m+1$.

Next we need to check the elliptic property (\ref{quasielliptic}). By construction, $g^m_1(z,\tau)$ is invariant under $z\mapsto z+\mu$. It is thus enough to check $z\mapsto z+\lambda \tau$.  Let us briefly review how this works for derivatives $\partial_\tau$. We will only discuss the case of index 0. If $f(z,\tau)$ satisfies (\ref{quasielliptic}), then we have for $F(z,\tau)=\partial_\tau f(z,\tau)$
\be\label{deltauelliptic}
F(z+\lambda \tau, \tau) = (\partial_\tau f)(z+\lambda\tau,\tau)
= \frac{d}{d\tau} f(z+\lambda\tau,\tau) - \lambda \partial_z f(z+\lambda\tau, \tau)\ .
\ee
Here $\partial_\tau$ and $\partial_z$ denote the partial derivatives with respect to the second and first argument respectively, whereas $d/d\tau$ is the total derivative with respect to $\tau$. Because $f(z+\lambda\tau,\tau)$ is given by the right hand side of (\ref{quasielliptic}), we see immediately that the right hand side of (\ref{deltauelliptic}) is also of the form of (\ref{quasielliptic}).

Now let us turn to $g^m_1(z,\tau)$. We will proceed by induction. For the base case $m=1$ we have
\begin{multline}
g^1_1(z+\lambda\tau,\tau)=\int dz (\partial_\tau P_1)(z+\lambda\tau,\tau) = \int dz (d_\tau P_1 -\lambda \partial_z P_1)(z+\lambda\tau,\tau) \\
= \int dz \partial_\tau P_1(z,\tau) -\lambda (P_1(z+\lambda\tau,\tau)+ C(\tau))
= g^1_1(z,\tau) -\lambda (P_1(z,\tau)+C(\tau)) -\lambda^2 2\pi i
\end{multline}
where we used (\ref{Pkelliptic}), and $C(\tau)$ is the (in general $\tau$ dependent) integration constant.

For $g^m_1(z,\tau)$, we can apply this argument recursively using $g^{m+1}_1(z,\tau)= \int dz \partial_\tau g^m_1(z,\tau)$. 
We then have
\begin{multline}\label{gm1elliptic}
g^{m+1}_1(z+\lambda \tau, \tau)=  \int dz (\partial_\tau g^m_1)(z+\lambda\tau,\tau) =\int dz (d_\tau g^m_1 -\lambda \partial_z g^m_1)(z+\lambda\tau,\tau) \\
=\int dz d_\tau g^m_1(z+\lambda\tau,\tau) -\lambda (g^m_1(z+\lambda\tau,\tau)+ C(\tau))\ .
\end{multline}
By the induction assumption, $g^m_1(z+\lambda\tau,\tau)$ satisfies (\ref{quasielliptic}), which implies that last expression in (\ref{gm1elliptic}) is of the correct form.

\end{proof}

\section{Some Combinatorial Background}
\subsection{Stirling Numbers and Eulerian Numbers}\label{Stirling_appendix}

Our main recursion relation (Proposition \ref{prop:rec}) involves Stirling numbers of the first kind. In order to prove Proposition \ref{diffop}, which is used in the proof of Proposition \ref{prop:rec}, we also need Stirling numbers of the second kind. The proof of Proposition \ref{prop:reccomm} uses Eulerian numbers. In this section we review these combinatorial numbers as well as the related identities that we need in the proofs of our results. For more details, see e.g. \cite{MR2078910} and \cite{MR2868112}.

Given $n\in \mathbb{Z},$ $n>0,$ we define
\be (x)_n=x(x-1)\cdots (x-n+1).
\ee
\begin{defn}Given $n,k\in \mathbb{Z}$, define {\emph{the Stirling numbers of the first kind}} $s(n,k),$ to be the coefficients in the expansion of $(x)_n$ in powers of $x$, if $n>0$ and $k\ge 0$, that is,
\be (x)_n=\sum_{k=0}^ns(n,k)x^k,
\ee and define $s(0,0)=1$, and if $k,n\in \mathbb{Z}$ such that $k<0$ and $n\ge 0$, we define $s(n,k)=0$.
\end{defn}

\begin{defn}\label{second_kind}Given $n, k\in \mathbb{N}$, define {\emph{the Stirling numbers of the second kind}} $S(n,k)$ to be the coefficients given by
\be \sum_{k=0}^n S(n,k)(x)_k=x^n,
\ee if $n>0$ and $k\ge 0$, and define $S(0,0)=1$ and $S(n,0)=0$ if $n>0.$
\end{defn}

Definition \ref{second_kind} implies that $S(n,k)=0$ for $n,k\in\mathbb{N}$ such that $k>n.$

Before defining Eulerian numbers, we first recall the definition of a descent of a permutation. 

\begin{defn}Given a permutation $(i_1, i_2, \cdots, i_n)$ of $(1,2, \cdots, n),$ an ascent of $(i_1, i_2, \cdots, i_n)$ is a position where $i_j<i_{j+1}.$ Similarly, a descent of $(i_1, i_2, \cdots, i_n)$ is a position where $i_{j+1}>i_j.$ 
\end{defn} We will sometimes need permutations of subtuples $(s_1, s_2, \cdots, s_j)$ of $(1,2,\cdots, n).$ In this case, ascents and descents are defined in the obvious way.
\begin{defn}
    Given $n,k\in \mathbb{N},$ the {\emph{Eulerian number}} $A(n,k)$ is the number of permutations of $n$ with exactly $k$ descents, or equivalently, the number of permutations of $n$ with exactly $k$ ascents. 
\end{defn} We next give some identities satisfied by the Stirling numbers and the Eulerian numbers:

Let $n,k\in \mathbb{N}$. Then
\begin{equation}
S(n,k)=\frac{1}{k!}\sum_{j=0}^k(-1)^{k-j}\binom{k}{j}j^n
\end{equation} 
and
\begin{equation}\label{6}
S(n+1,k+1)=\sum_{j=0}^n\binom{n}{j}S(j,k).
\end{equation}Given $k,n\in \mathbb{N}$ such that $0<k\le n,$
\begin{equation}\label{5}
    S(n,k)=kS(n-1,k)+S(n-1,k-1),
\end{equation} and given $n, k\in \mathbb{N}$ with $n\ge k,$
\be \label{SA_eq}S(n,k)=\frac{1}{k!}\sum_{j=0}^{k-1}A(n,j)\binom{n-j-1}{k-j-1}.
\ee
 We will also need the identity
\be\label{Ss_eq} \sum_{j=k}^nS(n,j)s(j, k)=\delta_{n,k},
\ee for $n,k\in \mathbb{N}, k\le n$, where $\delta_{k,n}$ denotes the Kronecker delta.
The following identity will be used in the proof of Proposition \ref{prop:reccomm}.
\begin{prop}\label{identity_comm}
 Let $u,t\in \mathbb{Z},$ such that $u>0$ and $0\le t\le u.$ Then
    \begin{equation}
    \sum_{D=0}^{u-1}A(u,D)\sum_{i=t-D-1}^{u-D-1}\binom{u-D-1}{i}\frac{1}{(i+D+1)!}s(i+D+1, t)=\delta_{u,t}.
    \end{equation}
    \end{prop}
\begin{proof}Shifting the index $i$ to $i+D+1,$ we have \begin{equation*}
\sum_{D=0}^{u-1}A(u,D)\sum_{i=t-D-1}^{u-D-1}\binom{u-D-1}{i}\frac{1}{(i+D+1)!}s(i+D+1, t)=\sum_{D=0}^{u-1}A(u,D)\sum_{i=t}^{u}\binom{u-D-1}{i-D-1}\frac{1}{i!}s(i, t).\end{equation*}
Since $s(i,t)=0$ if $i<t,$ the right hand side of this equality is equal to
\begin{equation*}\sum_{D=0}^{u-1}A(u,D)\sum_{i=0}^{u}\binom{u-D-1}{i-D-1}\frac{1}{i!}s(i, t).\end{equation*} Applying Equation \eqref{SA_eq} followed by Equation \eqref{Ss_eq}, this is then equal to
\begin{equation*}\sum_{i=0}^us(i,t)\frac{1}{i!}\sum_{D=0}^{u-1}A(u,D)\binom{u-D-1}{i-D-1}\\=\sum_{i=0}^us(i,t)S(u,i)=\delta_{u,t}, 
\end{equation*}which proves the result.
\end{proof}

\subsection{Some combinatorial identities}
In this section, we prove several identities that are used in the proofs of our main results. In particular, this section includes Proposition \ref{diffop}, and several results required to prove it. Proposition \ref{diffop} is important in the proof of our main recursion relation (Proposition \ref{prop:rec}). 

\begin{lem}\label{deriv_iden}
Let $n\in \mathbb{Z}_+$. Then \be\label{diff_id}
\partial_\tau^{n} (1-q^k)^{-1} = \frac{q^k}{1-q^k} \sum_{r=0}^{n-1} (2\pi i k)^{n-r} \binom{n}{r} \partial_\tau^r (1-q^k)^{-1}.
\ee
\end{lem}
\begin{proof}
We induct on $n$. Equation \eqref{diff_id} is trivially true in the case that $n=1$. Suppose it is true for some $n\ge 1$. Then
\begin{multline*}
\partial_\tau^{n+1}(1-q^k)^{-1}=\partial_\tau(\partial_\tau^{n}(1-q^k)^{-1})=\partial_\tau\left( \frac{q^k}{1-q^k} \sum_{r=0}^{n-1} (2\pi i k)^{n-r} \binom{n}{r} \partial_\tau^r (1-q^k)^{-1}\right)\\
=\frac{q^k}{1-q^k}\left(\sum_{r=0}^{n-1} (2\pi i k)^{n-r} \binom{n}{r} \partial_\tau^{r+1} (1-q^k)^{-1}\right)+2\pi i k\left(\frac{q^k}{(1-q^k)^2}\right)\sum_{r=0}^{n-1} (2\pi i k)^{n-r} \binom{n}{r} \partial_\tau^r (1-q^k)^{-1}\\
=\frac{q^k}{1-q^k}\left(\sum_{r=0}^{n-1} (2\pi i k)^{n-r} \binom{n}{r} \partial_\tau^{r+1} (1-q^k)^{-1}
+2\pi i k
\left(1+\frac{q^k}{1-q^k}\right)\sum_{r=0}^{n-1} (2\pi i k)^{n-r} \binom{n}{r} \partial_\tau^r (1-q^k)^{-1}\right)\\
=\frac{q^k}{1-q^k}\left(\sum_{r=0}^{n-1} (2\pi i k)^{n-r} \binom{n}{r} \partial_\tau^{r+1} (1-q^k)^{-1}+\sum_{r=0}^{n-1} (2\pi i k)^{n+1-r} \binom{n}{r} \partial_\tau^r (1-q^k)^{-1}
+ (2\pi i k) \left(\partial_\tau^n(1-q^k)^{-1}\right)\right)\\
=\frac{q^k}{1-q^k}\left(\sum_{r=1}^{n} (2\pi i k)^{n+1-r} \binom{n}{r-1} \partial_\tau^{r} (1-q^k)^{-1}+\sum_{r=0}^{n-1} (2\pi i k)^{n+1-r} \binom{n}{r} \partial_\tau^r (1-q^k)^{-1}
+ (2\pi i k) \left(\partial_\tau^n(1-q^k)^{-1}\right)\right)\\
=\frac{q^k}{1-q^k}\left(\sum_{r=0}^{n-1} (2\pi i k)^{n+1-r} \binom{n+1}{r} \partial_\tau^{r} (1-q^k)^{-1}
+ (2\pi i k)(n+1) \left(\partial_\tau^n(1-q^k)^{-1}\right)\right)\\
=\frac{q^k}{1-q^k}\left(\sum_{r=0}^{n} (2\pi i k)^{n+1-r} \binom{n+1}{r} \partial_\tau^{r} (1-q^k)^{-1}\right)\ 
\end{multline*}
where we used the Pascal triangle identity $\binom{n}{r-1}+\binom{n}{r}=\binom{n+1}{r}$.
\end{proof}

\begin{prop}\label{A-Stirling_2}
For all $k,m\in \mathbb{Z}$ such that $k\ne 0$ and $m\ge 0$,
\be\label{A}
\partial_\tau^m (1-q^k)^{-1}
= (2\pi i k)^m \sum_{i=0}^m i!S(m,i) \frac{1}{1-q^k}\left(\frac{q^k}{1-q^k}\right)^i.
\ee 
\end{prop}
\begin{proof}
We proceed by induction in $m$. The base case $m=0$ is obvious. For $m\ge1$ we use lemma~\ref{deriv_iden} to write
\begin{equation}
\partial_\tau^{m}(1-q^k)^{-1}
=\frac{(2\pi ik)^m}{1-q^k}\sum_{r=0}^{m-1}\binom{m}{r}\sum_{i=0}^ri!S(r,i)\left(\frac{q^k}{1-q^k}\right)^{i+1}
=\frac{(2\pi ik)^m}{1-q^k}\sum_{r=0}^{m-1}\sum_{i=0}^{m-1}\binom{m}{r}i!S(r,i)\left(\frac{q^k}{1-q^k}\right)^{i+1}
\end{equation} 
where we used the fact that $S(m,n)=0$ if $m<n$.
Applying Equation \ref{6} and then Equation \ref{5}, we have
\begin{equation}\label{recursS}
\sum_{r=0}^{m-1}\binom{m}{r}S(r,i)=S(m+1,i+1)-S(m,i)=(i+1)S(m,i+1).
\end{equation} 
This gives
\be
\partial_\tau^{m}(1-q^k)^{-1}
=\frac{(2\pi ik)^m}{1-q^k}\sum_{i=0}^{m-1}(i+1)!S(m,i+1)\left(\frac{q^k}{1-q^k}\right)^{i+1}
= \frac{(2\pi ik)^m}{1-q^k}\sum_{i=0}^{m}i!S(m,i)\left(\frac{q^k}{1-q^k}\right)^{i}
\ee
where we used the fact that $S(m,0)=0$ for $m>0$.

\end{proof}

\begin{prop}\label{diffop}
For all $k,l\in \mathbb{Z}$ such that $k\ne 0$ and $l\ge 0$,
\be\label{B}
\frac{1}{1-q^k}\left(\frac{q^k}{1-q^k}\right)^l
= \sum_{m=0}^l \frac{1}{l!}(2\pi ik)^{-m}s(l,m)\partial_\tau^m (1-q^k)^{-1}\ .
\ee 
\end{prop}
\begin{proof}
This follows from proposition~\ref{A-Stirling_2} and the fact that the Stirling numbers $s$ and $S$ are inverses of each other:
\begin{multline}\sum_{m=0}^l \frac{1}{l!}(2\pi ik)^{-m}s(l,m)\partial_\tau^m (1-q^k)^{-1}
=\frac1{l!}\sum_{m=0}^l \sum_{i=0}^l s(l,m) i!S(m,i)\frac{1}{1-q^k}\left(\frac{q^k}{1-q^k}\right)^i
\\=\frac1{l!}\sum^l_{i=0}i!\delta_{l,i}\frac{1}{1-q^k}\left(\frac{q^k}{1-q^k}\right)^i
= \frac{1}{1-q^k}\left(\frac{q^k}{1-q^k}\right)^l
\end{multline}
where we used that $S(m,i)=0$ if $m<i$ and $\sum_{m=0}^l s(l,m)S(m,i)=\delta_{l,i}$\ .

\end{proof}

\subsection{Partitions of permutations}\label{permutations_appendix}
Finally we introduce some identities on partitions of permutations.
\begin{defn}\label{Ct}
Let $r\in \mathbb{N}$ and let $\vec u$ be a permutation of a subtuple of $(1,\cdots, r)$. Define $C_{\vec u}$ to be
\begin{equation}C_{\vec u}=\sum_{p}\left(\frac{q^k}{1-q^k}\right)^{l_p}
\end{equation}where the sum is over all partitions $p$ of $\vec u$ into $l_p$ nonempty disjoint subtuples $\vec{u_1}\cup\cdots \cup {\vec u_{l_p}}=\vec u$ such that the entries in each subtuple are strictly increasing. If $\vec u=\emptyset,$ we define $C_{\vec u}=1.$
\end{defn}

As an example of definition \ref{Ct}, if we take $\vec u=(2, 3, 1, 4),$ then
\begin{equation*}
C_{(2,3,1,4)}=\left(\frac{q^k}{1-q^k}\right)^4+2\left(\frac{q^k}{1-q^k}\right)^3+\left(\frac{q^k}{1-q^k}\right)^2,
\end{equation*}since 
\begin{equation*}
(2, 3, 1, 4)=(2)\cup (3)\cup (1)\cup (4)=(2)\cup (3)\cup (1, 4)=(2,3)\cup (1)\cup (4)=(2,3)\cup (1,4).
\end{equation*}

\begin{lem}\label{LemCu}
Let $\vec u=(a_1,\cdots, a_n)$ be a permutation of a subtuple of $(1,\cdots, r)$ such that $a_1<\cdots<a_n.$ Then 
\begin{equation}\label{Cu}C_{\vec u}=\sum_{j=0}^{n-1}\binom{n-1}{j}\left(\frac{q^k}{1-q^k}\right)^{j+1}.
\end{equation}
\end{lem}
\begin{proof}We can partition ${\vec u}$ into a union of $j+1$ nonempty disjoint subtuples ${\vec u_1}\cup \cdots \cup {\vec u_l}$ by splitting ${\vec u}$ exactly $j$ times. Since the entries of ${\vec u}$ are strictly increasing, the subtuples in any partition of ${\vec u}$ also have increasing entries. Since there are $n-1$ choices for where to split ${\vec u}$ (after $a_1, \cdots, a_{n-2}$ or $a_{n-1}$), there are $\binom{n-1}{j}$ ways to partition ${\vec u}$ into a union of $j+1$ nonempty disjoint subtuples. The result follows.
\end{proof}

Since the formula given by equation \eqref{Cu} depends only on the length $n$ of the subtuple ${\vec u}$ and not on the entries of ${\vec u}$, we define
\begin{equation}\label{Cn}C_n=\sum_{j=0}^{n-1}\binom{n-1}{j}\left(\frac{q^k}{1-q^k}\right)^{j+1}.
\end{equation}

\begin{prop}\label{unique_part}
    Let ${\vec u}$ be a permutation of a subtuple of $(1,\cdots r)$. Then there is a unique partition of ${\vec u}$ into nonempty disjoint subtuples with increasing entries, such that the lengths of the subtuples are maximal. Let ${\vec u_1}\cup \cdots \cup {\vec u_n}={\vec u}$ be this partition for ${\vec u}$. For each $i$, $1\le i\le n$, let $r_i$ denote the length of ${\vec u_i}.$
    Then
    \begin{equation}\label{Cprod}
    C_{\vec u}=\prod_{i=1}^nC_{r_i},
    \end{equation}where $C_{r_i}$ is given by Equation \eqref{Cn}.
\end{prop}
\begin{proof}In order to prove the first statement of the proposition, we observe that a partition as described in the statement of the proposition is formed by splitting ${\vec u}$ a minimal number of times: Given ${\vec u}=(a_1,\cdots, a_r),$ we split ${\vec u}$ between $a_i$ and $a_{i+1}$ only when $a_{i}>a_{i+1}.$ Since we only split ${\vec u }$ where necessary to ensure the resulting subtuples are strictly increasing, the subtuples are of maximal length and any other partition of ${\vec u}$ is formed from this partition by splitting it further.

Let ${\vec u_1}\cup \cdots \cup {\vec u_n}={\vec u}$ be the partition of ${\vec u}$ described in the previous paragraph. Each $C_{\vec{u_i}}$ is a polynomial in $\frac{q^k}{1-q^k}$ such that the coefficient of $\left(\frac{q^k}{1-q^k}\right)^{j+1}$ is equal to the number of ways of splitting ${\vec u_i}$ into $j+1$ disjoint pieces. It follows that $\prod_{i=1}^nC_{r_i}$ is a polynomial in $\frac{q^k}{1-q^k}$ such that the coefficient of $\left(\frac{q^k}{1-q^k}\right)^j$ is equal to the number of ways of splitting $\vec{u}$ into $j$ increasing subtuples.
\end{proof}

\begin{prop}\label{Cprop}
Let $\vec u$ be a permutation of a nonempty subtuple of $(1,\cdots,r).$ Let $u$ denote the length of $\vec{u}$. Let $des(\vec u)$ denote the number of descents of $\vec u$, that is, if $\vec{u}=(u_1, \cdots, u_t)$ then $des(\vec u)$ is the number of $j$, $1\le j\le t-1$ such that $u_{j+1}<u_j$. Then
\begin{equation}
C_{\vec u}=\sum_{i=0}^{u-des(\vec u)-1}\binom{u-{des(\vec u)}-1}{i}\left(\frac{q^k}{1-q^k}\right)^{i+{des(\vec u)+1}}.
\end{equation}
\end{prop}
\begin{proof}Let $\vec{u}$ be a permutation of a subtuple of $(1,\cdots, r)$ and let $\vec{u}_1\cup \cdots\cup \vec{u}_n$ be the unique partition of $\vec{u}$ described in Proposition \ref{Cprop}, where the length of each ${\vec{u}_i}$ is denoted $r_i$. Observe that equation \eqref{Cn} in lemma \ref{LemCu} can be rewritten as 
\begin{equation}C_n=\left(1+\frac{q^k}{1-q^k}\right)^{n-1}\frac{q^k}{1-q^k}.
\end{equation} Plugging this equality into equation \ref{Cprod} in proposition \ref{unique_part}, we have
\begin{equation}\label{C_u_prod}
C_{\vec{u}}=\prod_{i=1}^{n}\left(\left(1+\frac{q^k}{1-q^k}\right)^{r_i-1}\frac{q^k}{1-q^k}\right).
\end{equation} Next observe that by the construction given in proposition \ref{unique_part}, the number $n$ of $\vec{u}_i$s is equal to $des(\vec{u})+1$, where $des(\vec{u})$ denotes the number of descents of $\vec{u}$. Also observe that $\sum_{i=1}^{des(\vec{u})+1}r_i=u,$ the length of $\vec{u}$. By equation \eqref{C_u_prod}, we therefore have
\begin{equation*}
C_{\vec{u}}=\left(1+\frac{q^k}{1-q^k}\right)^{u-des(\vec{u})-1}\left(\frac{q^k}{1-q^k}\right)^{des(\vec{u})+1}=\sum_{i=0}^{u-des(\vec{u})-1}\binom{u-des(\vec{u})-1}{i}\left(\frac{q^k}{1-q^k}\right)^{i+des(\vec{u})+1}.
\end{equation*}
\end{proof}

\section{Proof of the main recursion relation}
Here we collect various lemmas used in the proof of recursion relation proposition~\ref{prop:rec}.
\begin{lem}
For $k\neq 0$ we have 
\begin{multline}\label{receq1}
\zeta_1^{-k}F(b_0^{\vec s}a_k^1; (a^2, \zeta_2), \cdots, (a^n, \zeta_n); \tau)\\
=\frac{1}{1-q^k}\sum_{j=2}^n\left(\frac{\zeta_j}{\zeta_1}\right)^k\sum_{m=0}^\infty\binom{h_1-1+k}{m}F(b_0^{\vec s}; (a^2, \zeta_2), \cdots, (a_{m-h_1+1}^{1}a^j, \zeta_j), \cdots, (a^n, \zeta_n); \tau)\\
+\zeta_1^{-k}\frac{q^k}{1-q^k}\sum_{\emptyset \ne \vec t\subset \vec s}(2\pi i)^{t}F(b_0^{\vec s-\vec t}d_k^{\vec t}(a^1);(a^2, \zeta_2),\cdots, (a^n, \zeta_n);\tau)\ .
\end{multline}
\end{lem}
\begin{proof}
The proof is analogous to the proof of equation (B.17) in  \cite{Gaberdiel:2012yb}. 
Wlog we assume that $a^1$ is homogenous with weight $h_1$; the general case then follows by linearity.
The idea of the proof is to commute the mode $a^1_k$ all the way through to the right and then use cyclicity of the trace. The commutators can be expressed as
\be\label{comma}
[a_k^1, Y(\zeta_j^{L_0}a^j, \zeta_j)]=\sum_{m=0}^\infty\binom{h_1-1+k}{m}Y(\zeta_j^{L_0}a_{m-h_1+1}^1a^j, \zeta_j)\zeta_j^{k}\ .
\ee 
This gives
\begin{eqnarray*}
&&\zeta_1^{-k}F(b_0^{\vec s}a_k^1; (a^2, \zeta_2), \cdots, (a^n, \zeta_n); \tau)
= \zeta_1^{-k}\Tr {b_0^{\vec s}}  a_k^{1}Y(\zeta_2^{L_0} a^2, \zeta_2) \cdots Y(\zeta_n^{L_0}a^n,\zeta_n)q^{L_0}\\
&&=\zeta_1^{-k}\sum_{j=2}^n\Tr b_0^{\vec s}Y(\zeta_2^{L_0} a^2, \zeta_2)\cdots [a_k^{1}, Y(\zeta_j^{L_0} a^j, \zeta_j)]\cdots Y(\zeta_n^{L_0} a^n, \zeta_n)q^{L_0}\\
&&+\zeta_1^{-k}q^k\Tr  [a_k^{1}, b_0^{\vec s}]Y(\zeta_2^{L_0} a^2, \zeta_2) \cdots Y(\zeta_n^{L_0}a^n,\zeta_n)q^{L_0}
+\zeta_1^{-k}q^k\Tr b_0^{\vec s}  a_k^{1}Y(\zeta_2^{L_0} a^2, \zeta_2) \cdots Y(\zeta_n^{L_0}a^n,\zeta_n)q^{L_0}
\\
&& = \sum_{j=2}^n\left(\frac{\zeta_j}{\zeta_1}\right)^k\sum_{m=0}^\infty\binom{h_1-1+k}{m}F(b_0^{\vec s}; (a_2, \zeta_2), \cdots, (a_{m-h_1+1}^{1}a^j, \zeta_j), \cdots, (a^n, \zeta_n); \tau)
\\
&&+ \zeta_1^{-k}q^k\Tr  [a_k^{1}, b_0^{\vec s}]Y(\zeta_2^{L_0} a^2, \zeta_2) \cdots Y(\zeta_n^{L_0}a^n,\zeta_n)q^{L_0}
+q^k \zeta_1^{-k}F(b_0^{\vec s}a_k^1; (a^2, \zeta_2), \cdots, (a^n, \zeta_n); \tau)
 \end{eqnarray*}
In the second line we pick up all the commutators, the $q^k$ coming from the fact that $a^1_k q^{L_0}= q^k q^{L_0}a^1_k$, and the last equality comes from applying the expression for the commutator (\ref{comma}).
We then use lemma~\ref{lem:comm} and solve for $\zeta_1^{-k}F(b_0^{\vec s}a^1_k; (a^2, \zeta_2), \cdots, (a_n, \zeta_n);\tau)$  to obtain (\ref{receq1}).
\end{proof}

We are now ready to prove
the main lemma used in the proof of the recursion relation in proposition \ref{prop:rec}.
The lemma is an analogue of equation (B.16) in \cite{Gaberdiel:2012yb}. 
\begin{lem}\label{genLem}
Given $k\in \mathbb{Z}\setminus \{0\}$,
\begin{multline}\label{B16_analogue}
\zeta_1^{-k}F(b_0^{(1,2,\cdots,r)}a_k^1; (a^2, \zeta_2), \cdots, (a^n, \zeta_n);\tau)=\\\sum_{j=2}^n\left(\frac{\zeta_j}{\zeta_1}\right)^k\sum_{\vec s\subset (1, 2,\cdots, r)}\sum_{\vec u}(2\pi i)^u\frac{1}{1-q^k}C_{\vec u}\sum_{m=0}^\infty\binom{h-1+k}{m}F(b_0^{\vec s};(a^2, \zeta_2), \cdots, (d^{\vec u}_{m-h+1}(a^1)a^j, \zeta_j),\cdots, (a^n, \zeta_n);\tau).
\end{multline}
Here, the $\vec s$ sum is over all subtuples $\vec s$ of $(1, 2, \cdots, r)$, including $\emptyset$. The $\vec u$ sum is over all permutations $\vec u$ of $(1, 2, \cdots, r)- \vec s,$ $u$ denotes the length of $\vec{u}$, and $C_{\vec u}$ is given by definition \ref{Ct}. Here, $h$ denotes the weight of $d^{\vec u}(a^1)$ and the expression is expanded if needed in the case that $d^{\vec u}(a^1)$ is of inhomogeneous weight.
\end{lem}
\begin{proof}We prove the lemma by induction on $r$. The case $r=0$ follows immediately from equation \eqref{receq1}. Suppose the lemma holds of all $l\le r$ for some $r\ge 0$.

By equation \eqref{receq1}, 
\begin{multline}
\zeta_1^{-k}F(b_0^{(1,2,\cdots, r+1)}a_k^1; (a^2, \zeta_2),\cdots, (a^n, \zeta_n);\tau)\\=\frac{1}{1-q^k}\sum_{j=2}^n\left(\frac{\zeta_j}{\zeta_1}\right)^k\sum_{m=0}^\infty \binom{h_1-1+k}{m}F(b_0^{(1, 2, \cdots, r+1)};(a^2 \zeta_2), \cdots, (a^1_{m-h_1+1}a^j, \zeta_j),\cdots, (a^n, \zeta_n);\tau)\\+\zeta_1^{-k}\frac{q^k}{1-q^k}\sum_{\emptyset\ne {\vec t}\subset {(1,2,\cdots, r+1)}}(2\pi i)^tF(b_0^{(1, 2, \cdots, r+1)-{\vec t}}d_k^{\vec t}(a^1);(a^2, \zeta_2), \cdots, (a^n, \zeta_n):\tau).
\end{multline}We then apply the inductive hypothesis to the second term on the right hand side of the above equality. This gives us

\begin{multline}\label{C011}
\zeta_1^{-k}F(b_0^{(1,2,\cdots, r+1)}a_k^1; (a^2, \zeta_2),\cdots, (a^n, \zeta_n);\tau)\\=\frac{1}{1-q^k}\sum_{j=2}^n\left(\frac{\zeta_j}{\zeta_1}\right)^k\sum_{m=0}^\infty \binom{h_1-1+k}{m}F(b_0^{(1, 2, \cdots, r+1)};(a^2 \zeta_2), \cdots, (a^1_{m-h_1+1}a^j, \zeta_j),\cdots, (a^n, \zeta_n);\tau)\\+\frac{q^k}{1-q^k}\sum_{\emptyset\ne {\vec t}\subset {(1,2,\cdots, r+1)}}(2\pi i)^t\sum_{j=2}^n\left(\frac{\zeta_j}{\zeta_1}\right)^k\sum_{\vec s\subset (1, 2,\cdots, r+1)-{\vec t}}\sum_{\vec u}(2\pi i)^u\frac{1}{1-q^k}C_{\vec u}\sum_{m=0}^\infty\binom{h-1+k}{m}\\\times F(b_0^{{\vec s}};(a^2, \zeta_2), \cdots, (d^{\vec{ u}\cup \vec{t}}_{m-h+1}(a^1)a^j, \zeta_j),\cdots, (a^n, \zeta_n);\tau),
\end{multline}
where the fourth sum in the second term on the right hand side is over all permutations ${\vec u}$ of $(1,2,\cdots, r+1)-{\vec t}-{\vec s}.$ The first term on the right hand side corresponds to the ${\vec s}=(1,2,\cdots, r+1)$ term in equation \eqref{B16_analogue}. We thus need to show that the second term is equal to 
\begin{multline*}
\sum_{j=2}^n\left(\frac{\zeta_j}{\zeta_1}\right)^k\sum_{\vec{s}\subsetneq (1, 2, \cdots, r+1)}\sum_{\vec{v}}(2\pi i)^v\frac{1}{1-q^k}C_{\vec{v}}\sum_{m=0}^\infty\binom{h-1+k}{m}\\\times F(b_0^{\vec{s}};(a^2,\zeta_2),\cdots, (d^{\vec{v}}_{m-h+1}(a^1)a^j, \zeta_j),\cdots, (a^n,\zeta_n);\tau),
\end{multline*}
where the third sum here is over all permutations $\vec{v}$ of $(1,2,\cdots, r+1)-{\vec s}$. 

To do this, we gather all terms in (\ref{C011}) for which $\vec u \cup \vec t=\vec v$.
We then see that proving the equality is equivalent to proving that given any such permutation $\vec{v}$ of $(1,2,\cdots, r+1)-{\vec s},$ 
\begin{equation}
C_{\vec v}=\frac{q^k}{1-q^k}\sum_{\emptyset\ne \vec{t}\subset(1,2,\cdots,r+1)}\sum_{\vec{u} : \vec{u}\cup\vec{t}={\vec v}} C_{\vec{u}},
\end{equation} 
where the second sum is over all permutations $\vec u$ of $(1,2,\cdots, r+1)-\vec t-\vec s$ such that $\vec u\cup \vec t=\vec v$. We first note that because $\vec t\neq \emptyset$, $\vec s\ne (1,2,\cdots, r+1).$ This means that if $\vec v$ is a permutation of $(1,2, \cdots, r+1)-\vec s,$ $\vec v\ne \emptyset.$ Therefore, any partition $\vec {v_1}\cup \cdots \cup \vec {v_l}$ of $\vec v$ into nonempty subtuples $\vec {v_i}$ with strictly increasing entries has at least one suptuple. That is, any such partition $\vec {v_1}\cup \cdots \cup \vec {v_l}$ of $\vec v$ is such that $l\ge 1$. It follows that any such partition is equal to the union of a partition of a (possibly empty) subtuple $\vec u$ of $\vec v$ into subtuples with strictly increasing entries, and a nonempty subtuple $\vec t$ of $(1,2,\cdots, r+1)$, such that $\vec v=\vec u\cup \vec t.$ Recall that $C_{\vec v}$ is the sum of terms $\left(\frac{q^k}{1-q^k}\right)^l$ corresponding to each partition $\vec {v_1}\cup \cdots \cup \vec {v_l}$ of $\vec u$ into subtuples with strictly increasing entries. By the above argument, if $\vec u$ is such that $\vec v=\vec u\cup \vec t,$ where $\vec t$ is a subtuple of $(1,2,\cdots, r+1)$, $\frac{q^k}{1-q^k}$ times any term in $C_{\vec u}$ is a term in $\vec v$. On the other hand, any term in $C_{\vec v}$ is $\frac{q^k}{1-q^k}$ times a term in $C_{\vec u}$ for some (possibly empty) $\vec u$ such that $\vec v=\vec u\cup \vec t.$ The result follows.

\end{proof}

\def\cprime{$'$}


\end{document}